\documentclass{amsart}
\usepackage{graphicx}
\usepackage{amssymb}
\usepackage{amsmath}
\usepackage{amsthm}
\usepackage{enumerate}
\usepackage{bbm}
\usepackage{caption}
\usepackage{subcaption}
\usepackage{calligra}
\usepackage{MnSymbol}
\usepackage{bbm}
\usepackage{yfonts}
\DeclareMathAlphabet{\mathcalligra}{T1}{calligra}{m}{n}

\newtheorem{thm}{Theorem}[section]
\newtheorem{cor}[thm]{Corollary}

\newtheorem{lem}[thm]{Lemma}
\newtheorem{prop}[thm]{Proposition}
\theoremstyle{definition}
\newtheorem{defn}[thm]{Definition}

\newtheorem*{defn*}{Definition}
\newtheorem*{rems*}{Remarks}
\newtheorem*{rem*}{Remark}

\numberwithin{equation}{section}

\newcommand{\M}{{M}}

\newcommand{\C}{{C}}

\begin{document}

\title[The Singular Evolutoids Set and the Extended Evolutoids Front] {The Singular Evolutoids Set and the Extended Evolutoids Front}
\author{Micha\l{} Zwierzy\'nski}
\address{Warsaw University of Technology\\
Faculty of Mathematics and Information Science\\
ul. Koszykowa 75\\
00-662 Warsaw, Poland}

\email{zwierzynskim@mini.pw.edu.pl}



%%%%%%%%%%%%%%%%%%%%%%%%%%%%%%%%%%%%
%%%%%%%%% COMANDS %%%%%%%%%%%%%%%%%%
%%%%%%%%%%%%%%%%%%%%%%%%%%%%%%%%%%%%
\renewcommand{\C}{\mathcal{C}}
\newcommand{\Ev}{\mathcal{E}}%\textgoth{e}}%\mathrm{E}}
\newcommand{\SES}{\mathrm{SES}}
\newcommand{\EEF}{\mathbb{E}}
\newcommand{\e}{\mathrm{e}}
\renewcommand{\i}{\mathrm{i}}
\renewcommand{\tt}{\mathbbm{t}}
\newcommand{\nn}{\mathbbm{n}}
\renewcommand{\d}{\mathrm{d}}
\newcommand{\ff}{\textgoth{F}}

\begin{abstract}
In this paper we introduce the notion of the singular evolutoid set which is the set of all singular points of all evolutoids of a fixed smooth planar curve with at most cusp singularities. By the Gauss-Bonnet Theorem for Coherent Tangent Bundles over Surfaces with Boundary (Theorem 2.20 in \cite{DZ-GaussBonnet}) applied to the extended front of evolutoids of a hedgehog we obtain an integral equality for smooth periodic curves.
\end{abstract}

\maketitle

%%%%%%%%%%%%%%%%%%%%%%%%
%%%%% Introduction %%%%%%%%%%%%%
%%%%%%%%%%%%%%%%%%%%%%%%
\section{Introduction}

\noindent For a planar regular curve $\C$, let us fix an angle $\alpha$ and consider the envelope of straight lines obtained by rotating each tangent line to $\C$ at some $p\in\C$ about $p$ counterclockwise through $\alpha$. This envelope is called an $\alpha$-evolutoid of $\C$ (see \cite{AAGJ1, GiblinWarder, Hamman, JC, JARVY} and the literature therein). For some values of $\alpha$, an $\alpha$-evolutoid admits singular points. In this paper we define the singular evolutoids set of $\C$ ($\SES(\C)$) as the set of all singular points of all evolutoids of $\C$. We study its geometry and we relate this set to the singular set of extended evolutoids front of $\C$, i.e. the union of all $\alpha$-evolutoids for $\alpha\in[0,\pi]$, each embedded into its own slice of some extended space. We also study the geometry of these sets when $\C$ is a hedgehog -- a curve that can be parameterized using its Gauss map (on the theory of hedgehogs see \cite{MMY1, MMY2} and the literature therein).

The extended evolutoids front is an example of a front. The geometry of fronts and their generalizations -- coherent tangent bundles -- has been studied recently in \cite{DZ-GaussBonnet,MS1, MSUY1, SUY, SUY2, SUY3, SUY4}. In \cite{DZ-GaussBonnet} the authors generalize results in \cite{SUY, SUY2} to the following Gauss-Bonnet-type formulas (Theorem 2.20 in \cite{DZ-GaussBonnet}) for the singular coherent tangent bundle $\mathcal{E}$ over a smooth compact-oriented surface $M$ with boundary whose set of singular points $\Sigma$ admits at most peaks and $\Sigma$ is transversal to the boundary $\partial M$:
\begin{align}
\label{GBplusformula}2\pi\chi(M)&=\int_{M}K\d A+2\int_{\Sigma}\kappa_s\d \tau \\
\nonumber &+\int_{\partial M\cap M^+}\hat{\kappa}_g\d \tau-\int_{\partial M\cap M^-}\hat{\kappa}_g\d \tau-\sum_{p\in \text{null}(\Sigma\cap\partial M)}(2\alpha_+(p)-\pi),\\
\label{GBminusformula}\int_MK\d\hat{A}&+\int_{\partial M}\hat{\kappa}_g\d\tau = 2\pi\left(\chi(M^+)-\chi(M^-)\right)+2\pi\left(\# P^+-\# P^-\right)
\\ \nonumber &+\pi\left(\#(\Sigma\cap\partial M)^+ -\#(\Sigma\cap\partial M)^-\right)+\pi\left(\#P_{\partial M^+}-\#P_{\partial M^-}\right),
\end{align}
where $K$ is the Gaussian curvature, $\d A$ (respectively $\d\hat{A}$) is the unsigned (respectively signed) area form, $\d\tau$ is the arc length measure, $P^+$ (respectively $P^-$) is the set of positive (respectively negative) peaks in $M\setminus\partial M$, $(\Sigma\cap\partial M)^+$ (respectively $(\Sigma\cap\partial M)^-$, $\text{null}(\Sigma\cap\partial M)$) is the set of positive (respectively negative, null) singular points in $\Sigma\cap\partial M$, $P_{\partial M^+}$ (respectively $P_{\partial M^-}$) is the set of peaks in the positive (respectively negative) boundary. We will use the first formula to obtain an integral equality for smooth periodic curves in Corollary \ref{CorTheEquation}.

The geometry of the affine version of evolutoids and the affine version of extended front of evolutoids (the discriminant surface) was studied in \cite{CL-Affine}.

All the pictures in this manuscript were made in Wolfram Mathematica (\cite{Wolfram}).

%%%%%%%%%%%%%%%%%%%%%%%%%%%%%%%%%%%%%%%%%%%%%%%%%%%%%%%%%%%%%
%%%%%%%%%%%%% SECTION 1 %%%%%%%%%%%%%%%%%%%%%%%%%%%%%%%%%%%%%
%%%%%%%%%%%%%%%%%%%%%%%%%%%%%%%%%%%%%%%%%%%%%%%%%%%%%%%%%%%%%
\section{Preliminaries and Evolutoids}

\noindent Let $\C$ be a \textit{smooth planar curve}, that is the image of the $C^{\infty}$ map from an interval to $\mathbb{R}^2$. Let $s\mapsto f_\C$ be a parameterization of a fixed $\C$. A curve $\C$ is \textit{regular} if its velocity vector $f'_\C$ does not vanish. If $f'_{\C}(s_0)=0$, then the point $f_{\C}(s_0)$ is singular. The singular point $f(s_0)$ is called a cusp if it's locally diffeomorphic at $f(s_0)$ (in the source and in the target) to the curve $t\mapsto(t^2,t^3)$ at $t=0$. It's well known that a singular point $f(s_0)$ is a cusp if and only if the vectors $f''(s_0)$, $f'''(s_0)$ are linearly independent (e.g. see Theorem B.9.1 in \cite{UYBook}). A smooth curve is \textit{closed} if it is a $C^{\infty}$ map from $S^1$ to $\mathbb{R}^2$. A regular closed curve is \textit{locally convex} if its curvature does not vanish. A locally convex curve with rotation number equal to $m$ is called an \text{$m$-rosette} (see Figure \ref{fig:Fig_rosettes}). An \textit{oval} is a $1$-rosette. A regular point $f_{\C}(s_0)$ is called an \textit{inflexion point of $\C$} if its curvature changes sign at $s_0$ and is called an \textit{undulation point of $\C$} if its curvature vanishes at $s_0$ but does not change sign at $s_0$. An inflexion point $f(s_0)$ is \textit{non-degenerate} if $\det\big(f'_\C(s_0), f'''_\C(s_0)\big)\neq 0$. Furthermore, we will denote by $\kappa_{\C}(s)$, $\rho_{\C}(s)$, $\tt_{\C}(s)$, $\nn_{\C}(s)$ the curvature, the radius of curvature, the unit tangent, and the unit normal vector of $\C$ at $f_{\C}(s)$, where $\big(\tt_{\C}(s),\nn_{\C}(s)\big)$ forms a positive oriented frame. 

\begin{figure}[h]
    \centering
    \begin{subfigure}[h]{0.32\textwidth}
        \centering
        \includegraphics[width=\textwidth]{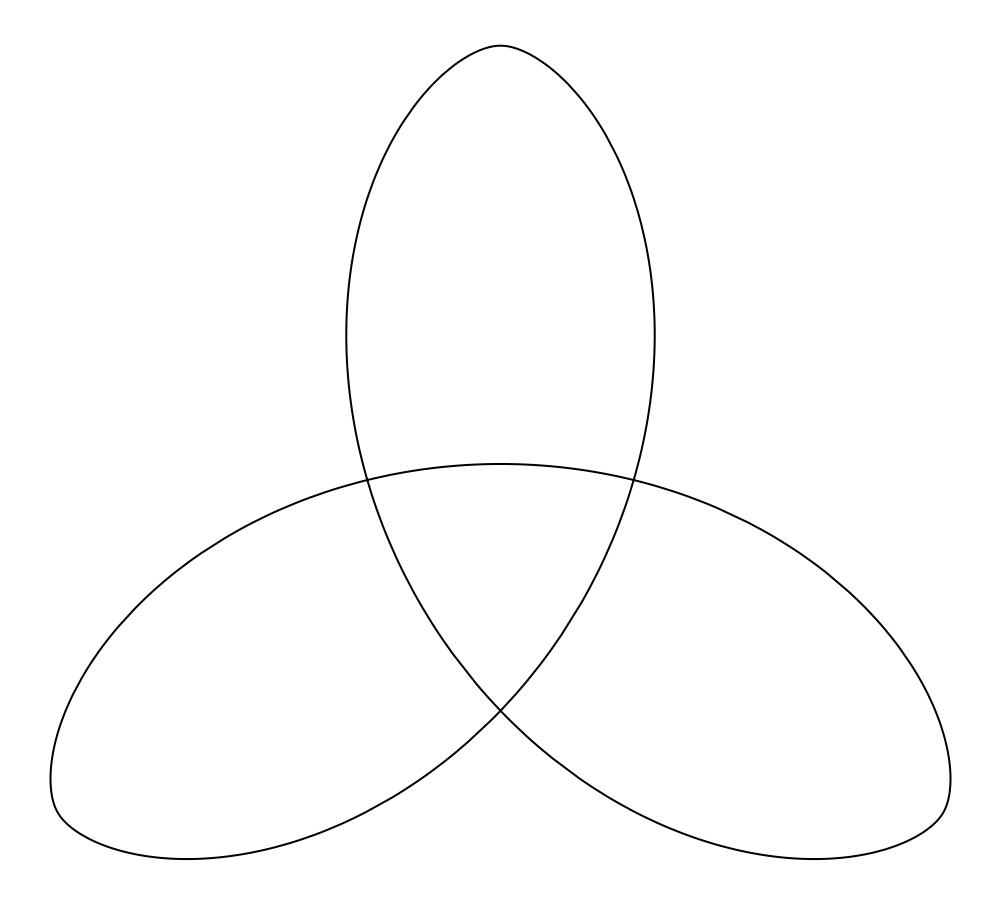}
        \caption{A $2$-rosette}
        \label{fig:Fig_rosettes01}
    \end{subfigure}
    \hfill
    \begin{subfigure}[h]{0.32\textwidth}
        \centering
        \includegraphics[width=\textwidth]{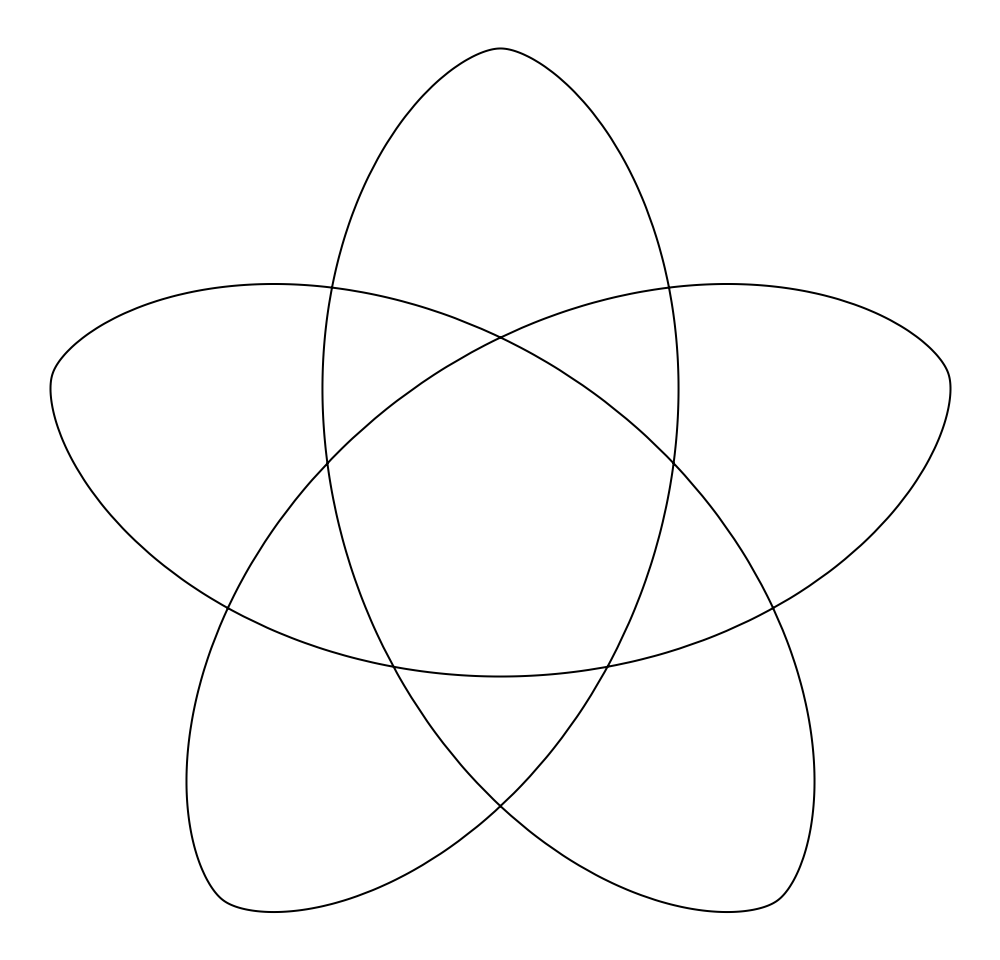}
        \caption{A $3$-rosette}
        \label{fig:Fig_rosettes02}
    \end{subfigure}
    \hfill
    \begin{subfigure}[h]{0.32\textwidth}
        \centering
        \includegraphics[width=\textwidth]{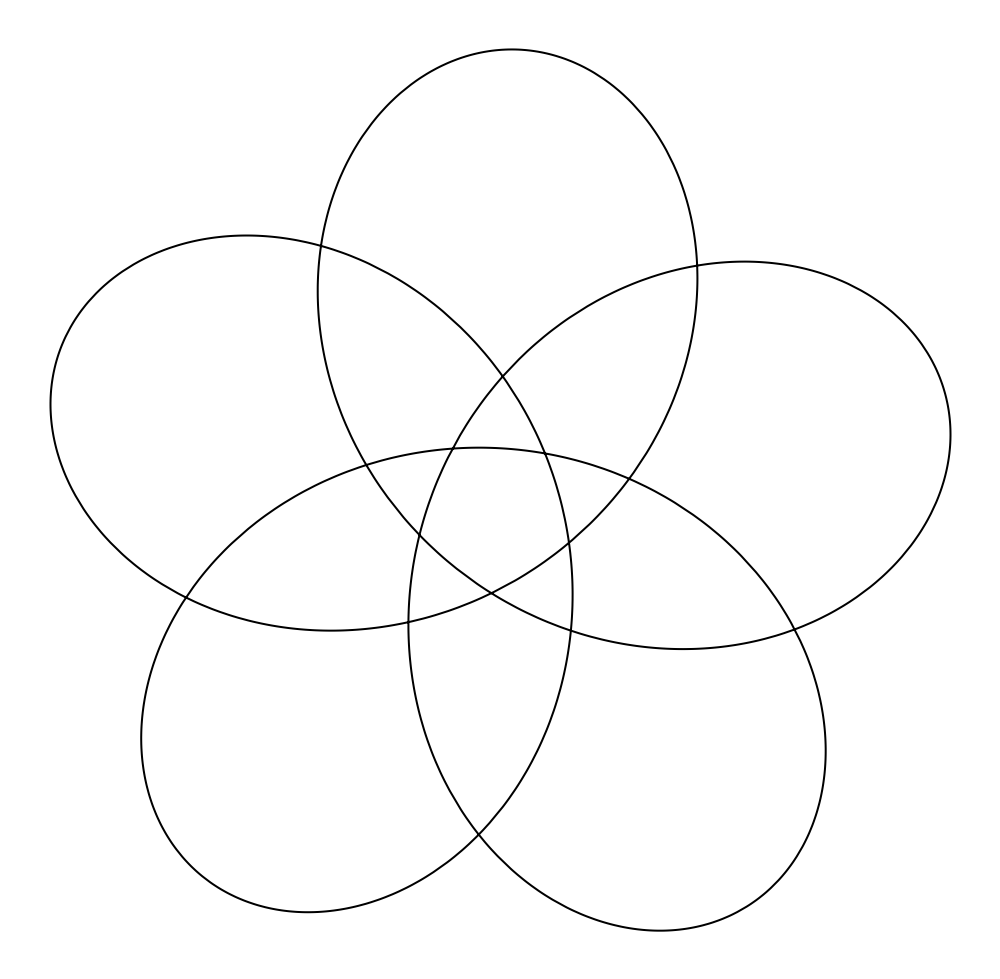}
        \caption{A $4$-rosette}
        \label{fig:Fig_rosettes03}
    \end{subfigure}
    \caption{Rosettes}
    \label{fig:Fig_rosettes}
\end{figure}
 
\begin{defn}\label{DefEvolutoid}
Let us fix $\alpha\in[0,\pi]$ and a smooth regular curve $\C$. For a fixed $s_0$ in the domain of $f_\C$ we set $\ell_{\alpha,s}$ as the tangent line of $\C$ at $f_{\C}(s_0)$ rotated by the angle $\alpha$ around $f(s_0)$. Then an \textit{$\alpha$-evolutoid of }$\C$, $\Ev_{\alpha}(\C)$, is the envelope of the family of $\ell_{\alpha,s}$ for all $s$ in the domain of $f_{\C}$.
\end{defn}

By Definition \ref{DefEvolutoid} the evolute of $\C$ is just a $\frac{\pi}{2}$-evolutoid of $\C$. Note that $\Ev_0(\C)=\Ev_{\pi}(\C)$ is $\C$ together with the tangent lines at inflexion points (\cite{GiblinWarder}). For a fixed value of $\alpha\in(0,\pi)$, a parameterization of $\Ev_{\alpha}(\C)$ is 
\begin{align}
\label{eqACparameter} \gamma_{\alpha}(s):&=f_{\C}(s)+\rho_{\C}(s)\sin\alpha\cdot\e^{\i\alpha}\tt_{\C} (s)\\ 
\nonumber &=f_{\C}(s)+\rho_{\C}(s)\sin\alpha\cos\alpha\cdot \mathbbm{t}_{\C}(s)+\rho_{\C}(s)\sin^2\alpha\cdot\mathbbm{n}_{\C}(s),
\end{align}
where $\e^{\i\alpha}$ is a rotation matrix by the angle $\alpha$. Whenever we talk about singular points of an $\alpha$-evolutoid, we will understand it as a singular point of the parameterization \eqref{eqACparameter} (or an equivalent one). In the following proposition we present basic geometric properties of evolutoids (see \cite{GiblinWarder, Hamman} for details).

\begin{prop}\cite{GiblinWarder, Hamman}\label{PropProperties}
Let $\alpha\in(0,\pi)$ and let $\C$ be a smooth regular curve with the arc length parameterization $s\mapsto f_{\C}(s)$ which is not locally a straight line. Let $\gamma_{\alpha}$ be defined as in \eqref{eqACparameter}. Then
\begin{enumerate}[(i)]
    \item if $f_{\C}(s_0)$ is an inflexion or an undulation point of $\C$, then the line $\ell_{\alpha, s}$ is an asymptote of $\Ev_{\alpha}(\C)$.
    \item the point $\gamma_{\alpha}(s)$ is singular if and only if $\rho'_{\C}(s)=-\cot\alpha$.
    \item the singular point $\gamma_{\alpha}(s)$ is a cusp if and only if $\rho''_{\C}(s)\neq 0$.
    \item $\gamma_{\alpha}$ does not have any inflexion points or undulation points.
    \item the subset of smooth closed regular curves $f$ for which the $\alpha$-evolutoid of $f$ has at most cusp singularities is dense in $C^{\infty}(S^1,\mathbb{R}^2)$ with Whitney $C^{\infty}$ topology.
\end{enumerate}
\end{prop}

\begin{figure}[h]
    \centering
    \begin{subfigure}[h]{0.47\textwidth}
        \centering
        \includegraphics[width=\textwidth]{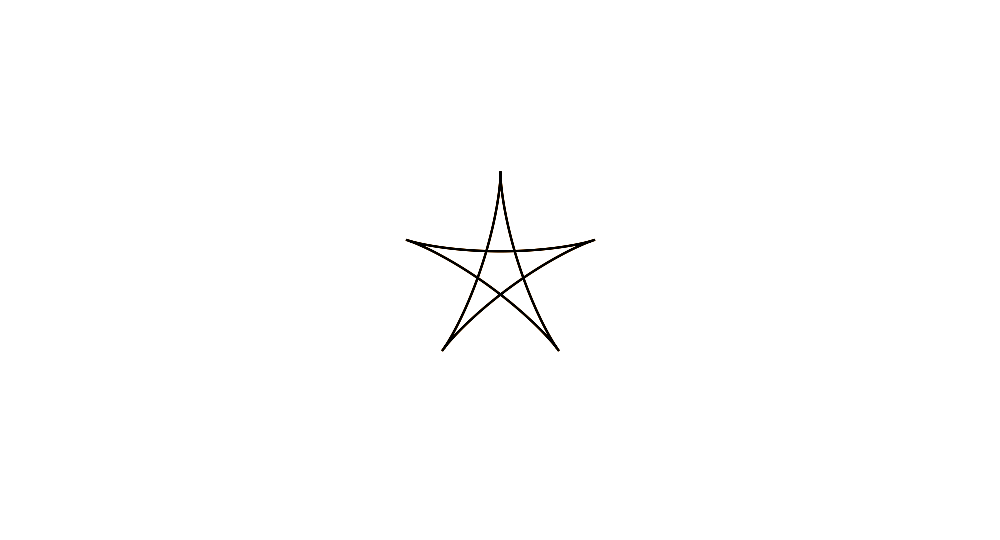}
        \caption{$\pentagram$}
        \label{fig:SESPentagram}
    \end{subfigure}
    \hfill
    \begin{subfigure}[h]{0.47\textwidth}
        \centering
        \includegraphics[width=\textwidth]{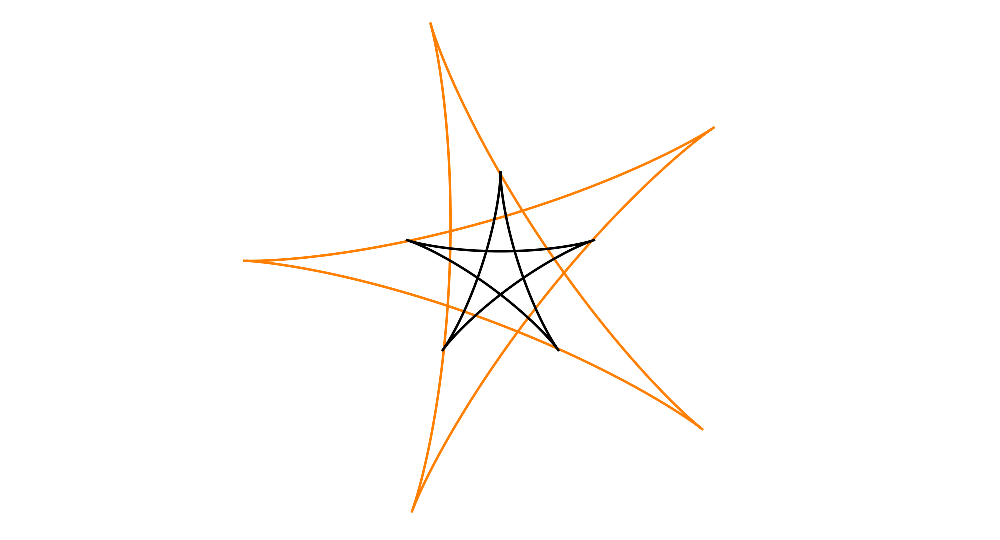}
        \caption{$\pentagram$ and $\Ev_{\pi/6}(\pentagram)$}
        \label{fig:SESPentagram}
    \end{subfigure}
    \\ 
    \begin{subfigure}[h]{0.47\textwidth}
        \centering
        \includegraphics[width=\textwidth]{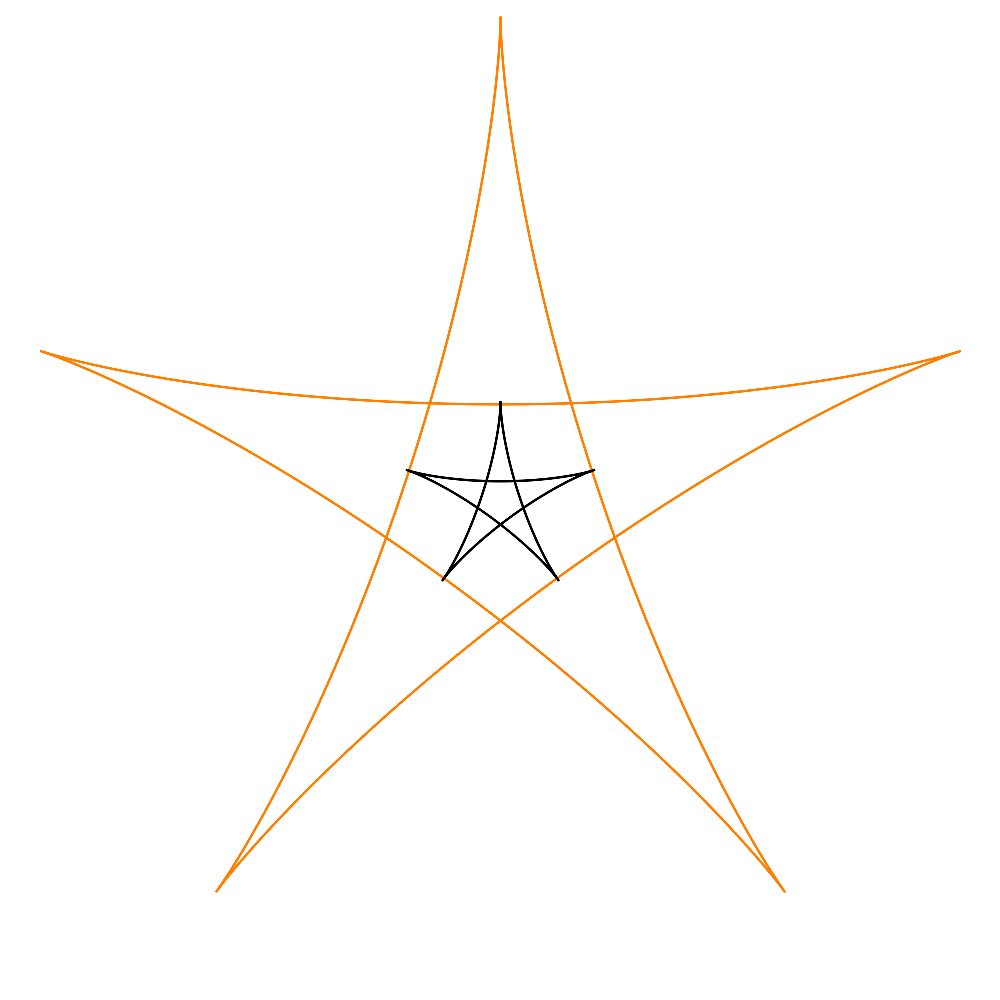}
        \caption{$\pentagram$ and $\Ev_{\pi/2}(\pentagram)$}
        \label{fig:SESPentagram}
    \end{subfigure}
    \hfill
    \begin{subfigure}[h]{0.47\textwidth}
        \centering
        \includegraphics[width=\textwidth]{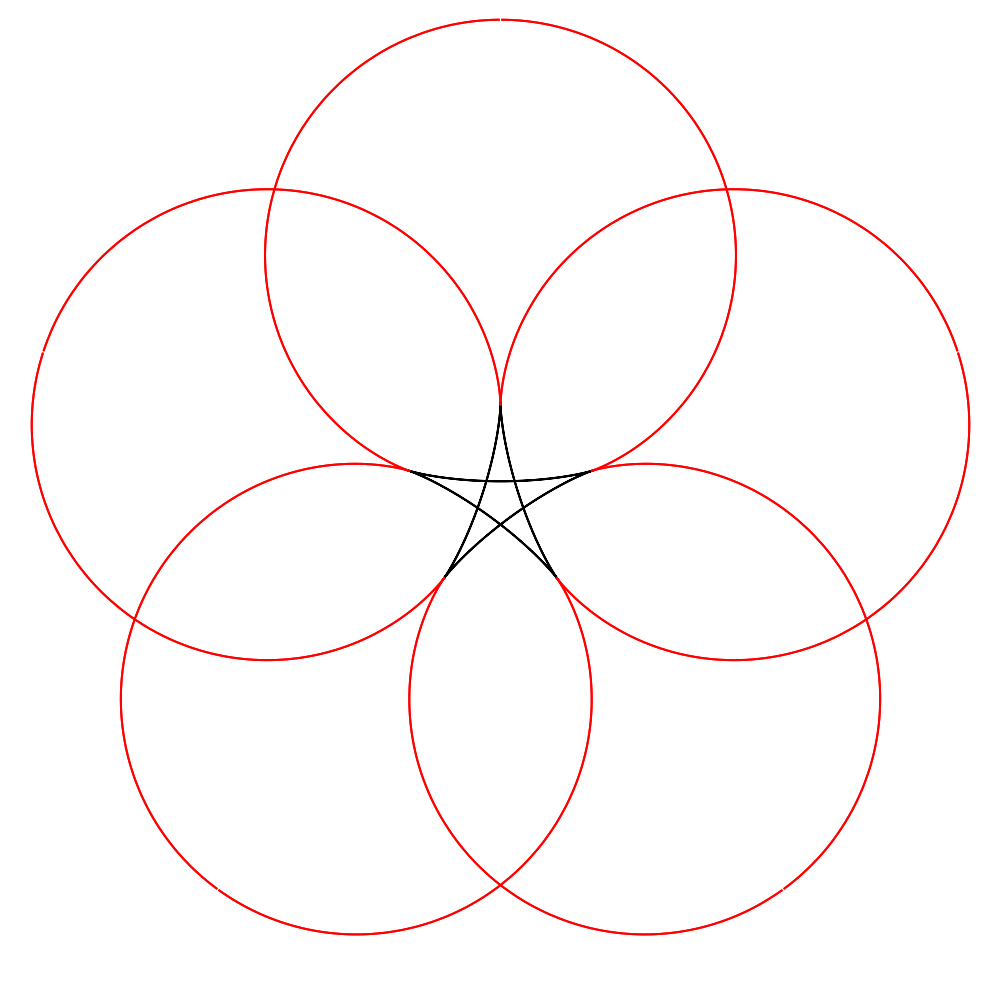}
        \caption{$\pentagram$ and $\SES(\pentagram)$}
        \label{fig:SESPentagram-D}
    \end{subfigure}
    \caption{A pentagram curve $\pentagram$ and its sets}
    \label{fig:SESPentagram}
\end{figure}

For a cusp singularity we can define a tangent line (as a limiting line in $\mathbb{R}P^1$). Therefore, we will extend Definition \ref{DefEvolutoid} for a class of smooth curves with at most cusp singularities. A curve at a cusp has a zero of its radius of curvature (also as a limiting point). Therefore, the parameterization of an $\alpha$-evolutoid given in \eqref{eqACparameter} still works. We illustrate a pentagram curve $\pentagram$ and its two evolutoids in Figure \ref{fig:SESPentagram}.

Now, let's recall the notion of a hedgehog and a support function. Let $\C$ be an oval of class $C^1$ in $\mathbb{R}^2$. The curve $\C$ can be viewed as the envelope of the family of its support lines given by
\begin{align}
\label{eq:FamilySupport}    x\cos\theta+y\sin\theta=p(\theta),
\end{align}
where the \textit{support function} $p$ is the signed distance of the support line to $\C$ with exterior normal vector $(\cos\theta,\sin\theta)$ from the fixed origin. Furthermore, for any $p$ at least class $C^2$ the envelope (even if degenerate) of family \eqref{eq:FamilySupport} exists and can be considered. Such curves are a special case of \textit{hedgehogs} -- curves that can be parameterized using their Gauss maps. Note that rosettes are non-singular hedgehogs. On the general theory of hedgehogs please refer to \cite{MMY1, MMY2} and the literature therein. 

The parameterization of a hedgehog in so-called polar tangential coordinates $(\theta,p(\theta))$ is
\begin{align}
    \label{eqSuppParamter}\theta\mapsto \big(p(\theta)\cos\theta-p'(\theta)\sin\theta, p(\theta)\sin\theta+p'(\theta)\cos\theta\big).
\end{align}
The usefulness of the support function, among other things, is that the basic geometric quantities of an oval can be calculated only using its support functions:
\begin{align}
    \kappa_{\C}(\theta)&=\dfrac{1}{p(\theta)+p''(\theta)},\\
    \rho_{\C}(\theta)&=p(\theta)+p''(\theta),\\
\label{eq:BlaschkeF}    L_{\C} &=\int_0^{2\pi}\rho_{\C}(\theta)\,\d\theta=\int_0^{2\pi}p(\theta)\,\d\theta,\\
\label{eq:CauchyF}    A_{\C}&=\dfrac{1}{2}\int_0^{2\pi}\big(p^2(\theta)-p'^2(\theta)\big)\,\d\theta,
\end{align}
where $L_{\C}$ and $A_{\C}$ denote the length and the area of $\C$, respectively. The formulas \eqref{eq:BlaschkeF} and \eqref{eq:CauchyF} are known as the Cauchy and the Blaschke formulas, respectively. Let $\C$ be a parameterized smooth curve. Its \textit{oriented area} (or \textit{algebraic area}) is defined as 
\begin{align*}
    \widetilde{A}_{\C}:=\dfrac{1}{2}\int_{\C}-y\,\d x+x\,\d y=\iint_{\mathbb{R}^2}\omega_{\C}(x,y)\, \d x\, \d y,
\end{align*}
where $\omega_{\C}(P)$ is the winding index of $\C$ around the point $P$. The oriented area of a hedgehog $\C$ with the rotation number $k$ ($k$-hedgehog) in terms of its support function is
\begin{align}
\label{eqOrientAreaDef}    \widetilde{A}_{\C}=\dfrac{1}{2}\int_0^{2k\pi}\big(p^2(\theta)-p'^2(\theta)\big)\,\d\theta.
\end{align}

Oriented areas of many curves, such that the evolute, the Wigner caustic, the constant width measure set, the secant caustic, isoptic curves, and the evolutoids made it possible to prove many isoperimetric inequalities and equalities associated with hedgehogs (e.g. see \cite{AAGJ1, DRZ1, DZ2022, JC, JARVY, MMYIneq, Zwierz1, Zwierz2, Zwierz3}). 

By direct observations one can get that the function
\begin{align}
\label{eqEvolSupport}    p_{\alpha}(\theta):=p(\theta-\alpha)\cos\alpha+p'(\theta-\alpha)\sin\alpha 
\end{align}
is a support function of the $\alpha$-evolutoid of a hedgehog $\C$ with support function $p$. 

\begin{prop}\label{PropAreaEvol}
Let $\C$ be a smooth hedgehog and let $\alpha\in [0,\pi]$. Then
\begin{align}
\widetilde{A}_{\Ev_{\alpha}(\C)}=\widetilde{A}_{\C}\cos^2\alpha+\widetilde{A}_{\Ev_{\pi/2}(\C)}\sin^2\alpha.
\end{align}
\end{prop}
\begin{proof}
Let $k$ be the rotation number of $C$. Then, by \eqref{eqOrientAreaDef} and \eqref{eqEvolSupport} and by direct calculations one can get the following relation
\begin{align}
    \widetilde{A}_{\Ev_{\alpha}(\C)}=\cos^2\alpha\cdot\dfrac{1}{2}\int_0^{2k\pi}\big(p^2(\theta)-p'^2(\theta)\big)\,\d\theta+\sin^2\alpha\cdot\dfrac{1}{2}\int_0^{2k\pi}\big(p'^2(\theta)-p''^2(\theta)\big)\,\d\theta.
\end{align}
Since the support function of the evolute of $\C$ ($\frac{\pi}{2}$-evolutoid) is $p'\big(\theta-\frac{\pi}{2}\big)$, we have completed the proof.
\end{proof}

It is well known that the oriented area of the evolute of a $1$-hedgehog $\C$ is non-positive and is equal to $0$ if and only if $\C$ is a circle. Therefore, by Proposition \ref{PropAreaEvol} we obtain the following small generalization of Theorem 1 in \cite{JC} (which states that when $\C$ and $\Ev_{\alpha}(\C)$ are ovals -- then their oriented areas are ordinary areas).

\begin{cor}
Let $\C$ be a smooth $1$-hedgehog and let $\alpha\in(0,\pi)$. Then
\begin{align*}
    \widetilde{A}_{\Ev_{\alpha}(\C)}\leqslant\widetilde{A}_{\C}\cos^2\alpha,
\end{align*}
and the equality holds if and only if $\C$ is a circle.
\end{cor}

\begin{figure}[h]
    \centering
    \begin{subfigure}[h]{0.41\textwidth}
        \centering
        \includegraphics[width=\textwidth]{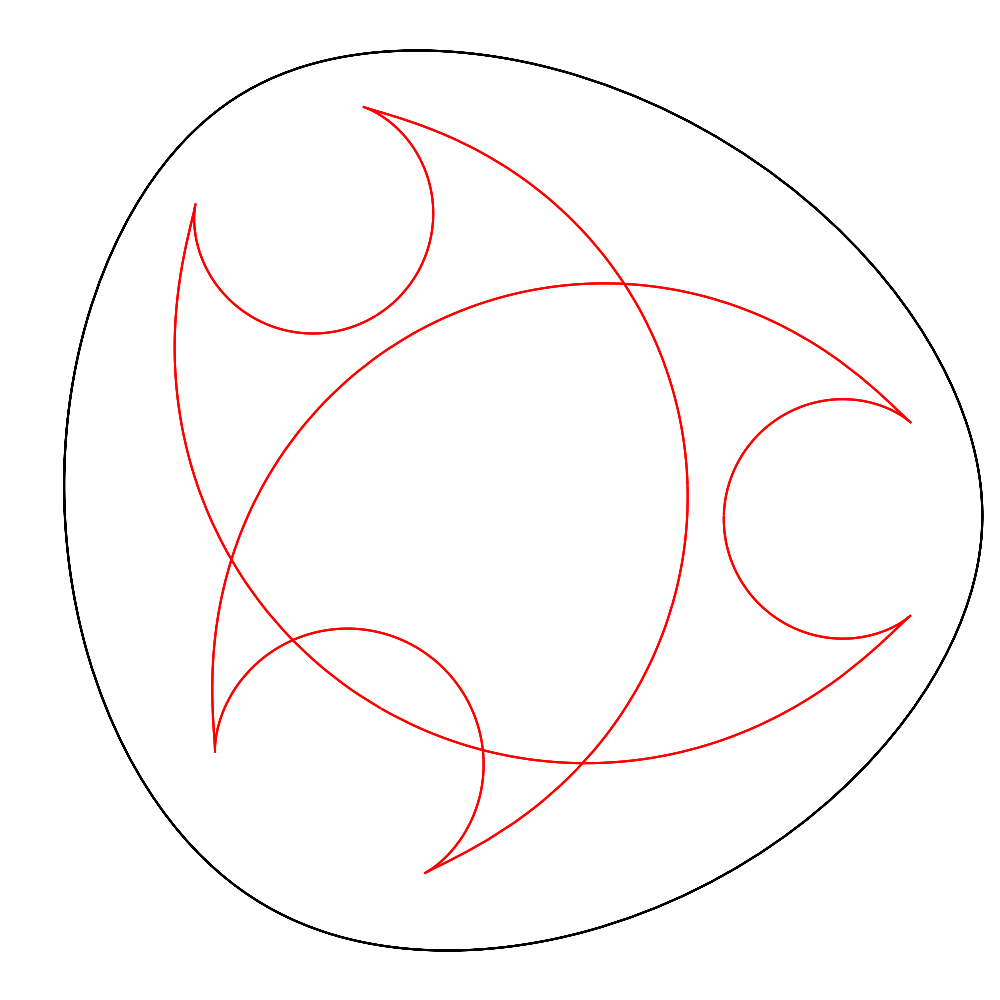}
        \caption{An oval $\C$ and $\SES(\C)$}
        \label{fig:SESFamily}
    \end{subfigure}
    \hfill
    \begin{subfigure}[h]{0.41\textwidth}
        \centering
        \includegraphics[width=\textwidth]{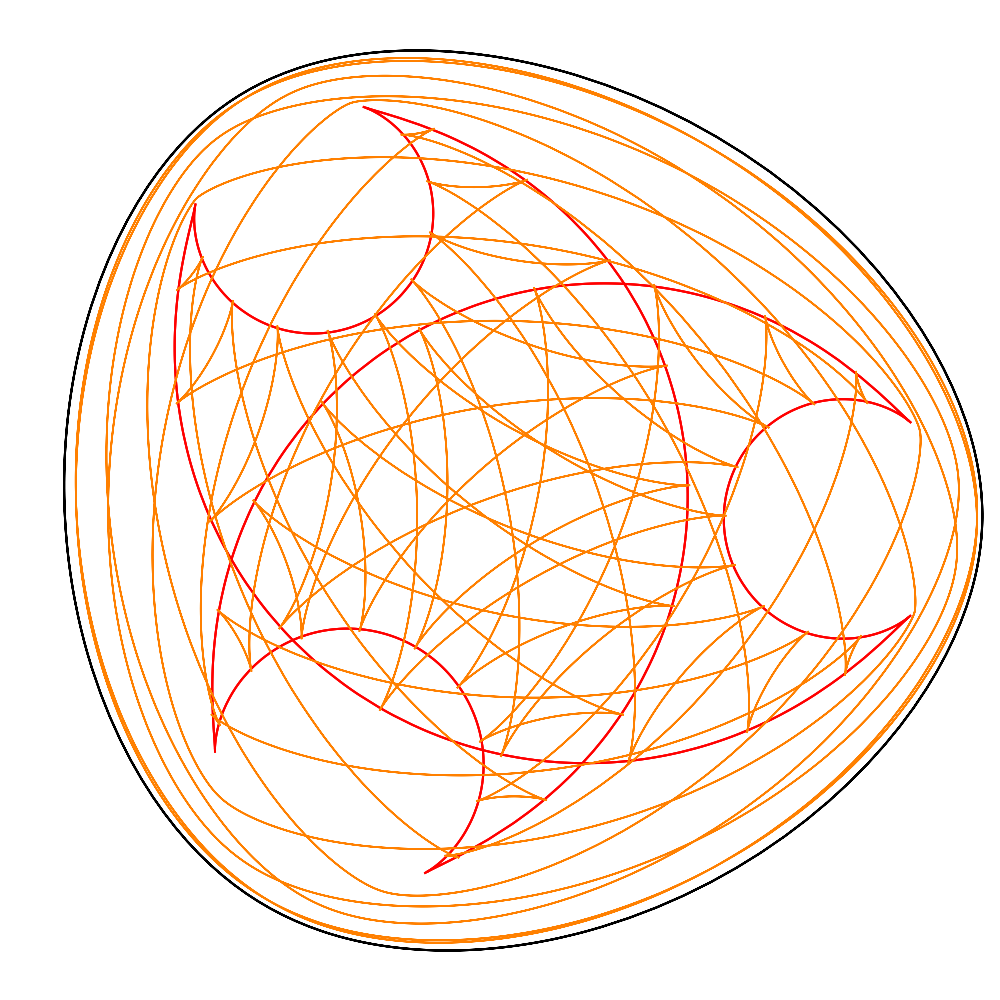}
        \caption{$\C$, $\SES(\C)$, and family of $\Ev_{\alpha}(\C)$}
        \label{fig:SESFamily}
    \end{subfigure}
    \\ 
    \begin{subfigure}[h]{0.41\textwidth}
        \centering
        \includegraphics[width=\textwidth]{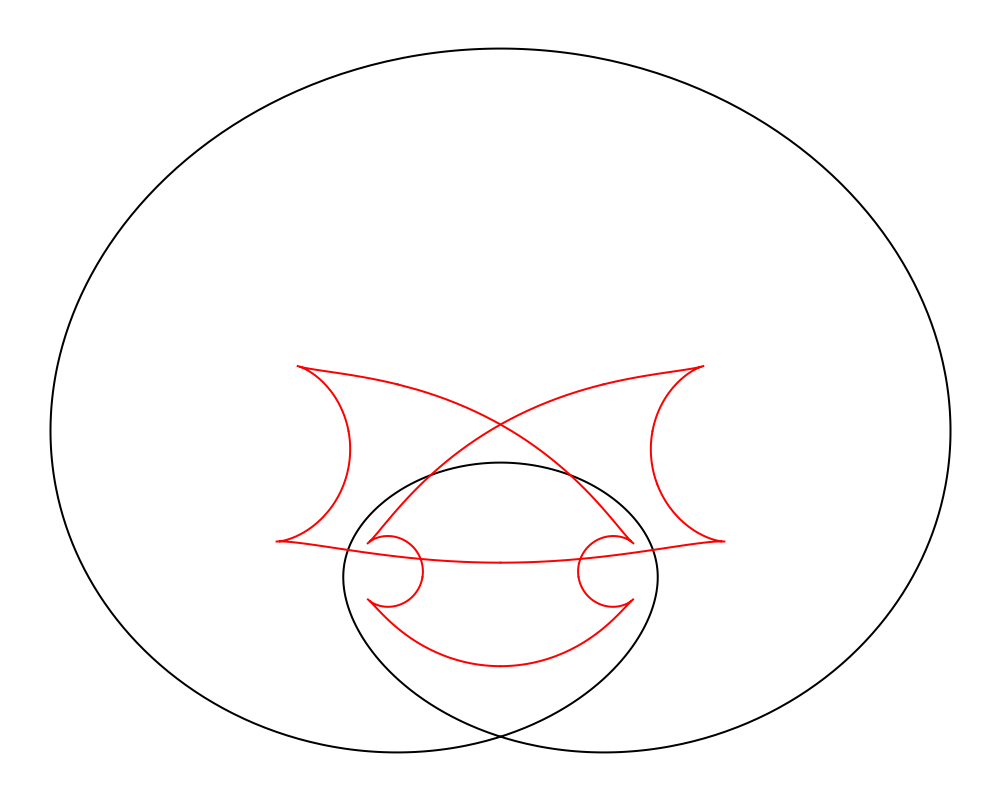}
        \caption{A $2$-rosette $\hat{\C}$ and $\SES(\hat{\C})$}
        \label{fig:SESFamily}
    \end{subfigure}
    \hfill
    \begin{subfigure}[h]{0.41\textwidth}
        \centering
        \includegraphics[width=\textwidth]{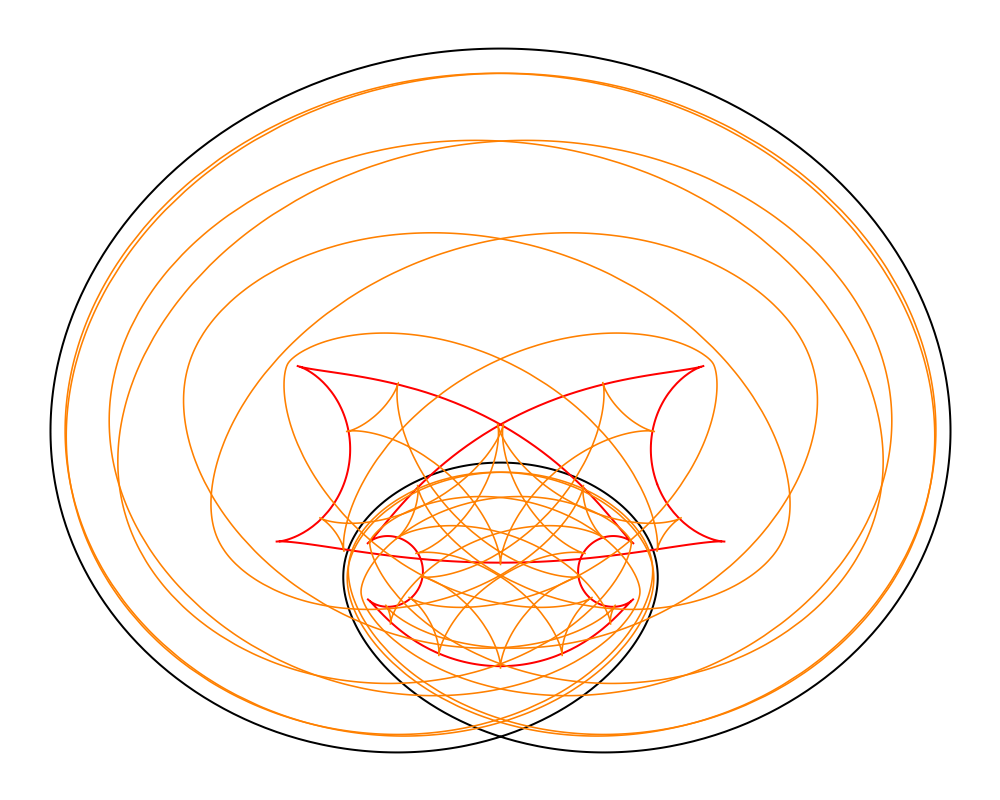}
        \caption{$\hat{\C}$, $\SES(\hat{\C})$, and family of $\Ev_{\alpha}(\hat{\C})$}
        \label{fig:SESFamily}
    \end{subfigure}
    \caption{Curves, their evolutoids, and the singular evolutoids sets}
    \label{fig:SESFamily}
\end{figure}
%%%%%%%%%%%%%%%%%%%%%%%%%%%%%%%%%%%%%%%%%%%%
%%%%%%%%%%%%%%%%%%%%%%%%%%%%%%%%%%%%%%%%%%%%
%%%%%%%%%%%%%%%%%%%%%%%%%%%%%%%%%%%%%%%%%%%%
\section{The Singular Evolutoids Set}

\noindent The evolute of $\C$ can be defined as the set of all singular points of offsets (or parallel curves) of $\C$. Similarly, an $\alpha$-evolutoid of $\C$ can be viewed as a set of all singular points of a suitable family of ‘skew parallel curves' (\cite{GiblinWarder}). In some similar way, we define the singular evolutoids set.

\begin{defn}\label{DefSES}
Let $\C$ be a smooth curve with at most cusp singularities. Then the \textit{singular evolutoids set of $\C$}, $\SES(\C)$, is the set of all singular points of a family of $\alpha$-evolutoids of $\C$ for $\alpha\in [0,\pi]$.
\end{defn}

In Figure \ref{fig:SESFamily} we illustrate an oval and a $2$-rosette together with their singular evolutoids sets and families of evolutoids. See also Figure \ref{fig:SESPentagram-D} for the singular evolutoids set of a singular curve. 

By Proposition \ref{PropProperties}(ii) an $\alpha$-evolutoid of $\C$ is singular at $f_{\C}(s)$ if and only if $\rho'_{\C}(s)=-\cot\alpha$. Therefore, by \eqref{eqACparameter} we get the following parameterization of the singular evolutoids set of $\C$:

\begin{align}
\label{EqSESParameter} 
\gamma_{\SES(\C)}(s):&= \gamma_{-\mathrm{arccot}(-
\rho_{\C}'(s))}(s)\\ 
\nonumber &=
f_{\C}(s)+\dfrac{-\rho_{\C}(s)\rho'_{\C}(s)\tt_{\C}(s)+\rho_{\C}(s)\nn_{\C}(s)}{1+\rho'^2_{\C}(s)}.
\end{align}

Similarly, like for evolutoid sets, whenever we write about singular points of the singular evolutoids set we will understand a singular point of the parameterization \eqref{EqSESParameter}. The following theorem describes geometric properties of the singular evolutoids set.

\begin{thm}\label{ThmSESProperties}
Let $\C$ be a smooth curve with at most cusp singularities and with the arc length parameterization (except for singular points) $s\mapsto f_{\C}(s)$ which is not locally a straight line. Let $\gamma_{\SES(\C)}$ be defined as in \eqref{EqSESParameter}. Then
\begin{enumerate}[(i)]
    \item the singular evolutoids set of $\C$ can be given by a smooth parameterization.
    \item if $f_{\C}(s)$ is not a cusp, then the singular evolutoids set of $\C$ is singular at $\gamma_{\SES(\C)}(s)$ if and only if $\rho''_{\C}(s)=0$ and the singular point is a cusp if and only if $\rho'''_{\C}(s)\neq 0$.
    \item if $f_{\C}(s)$ is a cusp of $\C$, then $\SES(\C)$ has a cusp at $f_{\C}(s)$.
    \item the singular evolutoids set of $\C$ has an inflexion point or an undulation point at $\gamma_{\SES(\C)}(s)$ if and only if $1+\rho'^2_{\C}(s)+2\rho_{\C}(s)\rho''_{\C}(s)=0,$ and the inflexion point is non-degenerate if and only if $\rho_{\C}'(s)+\rho_{\C}'^3(s)-\rho_{\C}^2(s)\rho_{\C}'''(s)\neq 0$.
    \item the subset of smooth closed curves $f$ with at most cusp singularities for which the singular evolutoids set has at most cusp singularities, no undulation points, and at most non-degenerate inflexion points is dense in $C^{\infty}(S^1,\mathbb{R}^2)$ with Whitney $C^{\infty}$ topology.
\end{enumerate}
\end{thm}
\begin{proof}
\begin{enumerate}[(i)]
\item It follows directly from the parameterization \eqref{EqSESParameter}. The parameterization is still correct even if $\C$ has an inflexion or undulation points. It can be seen using a Taylor expansion of $f_\C$.

\item If $f_{\C}(p)$ is not a cusp, it's easy to verify that $\gamma_{\SES(\C)}'(s)=0$ if and only if $\rho_{\C}''(s)=0$. Furthermore, the singular point $\gamma_{\SES(\C)}(s)$ is a cusp if and only if the vectors $\gamma_{\SES(\C)}''(s)$ and $\gamma_{\SES(\C)}'''(s)$ are linearly independent, which is equivalent to $\rho_{\C}'''(s)\neq 0$.

\item If $f_{\C}(p)$ is a cusp, then without loss of generality we can assume that $p=0$ and $f_{\C}$ is $t\mapsto (t^2,t^3)$. Then the singular evolutoids set has the following parameterization:
\begin{align}
\label{eq:SEScuspEq}    t\mapsto\left(t^2\cdot\dfrac{-8-54t^2+243t^4}{8+162t^2+648t^4}, t^3\cdot\dfrac{4-27t^2+81t^4}{4+81t^2+324t^2}\right),
\end{align}
which is also a cusp at $p=0$ (see Figure \ref{fig:sesCusp}).
\begin{figure}[h]
    \centering
    \includegraphics[scale=0.59]{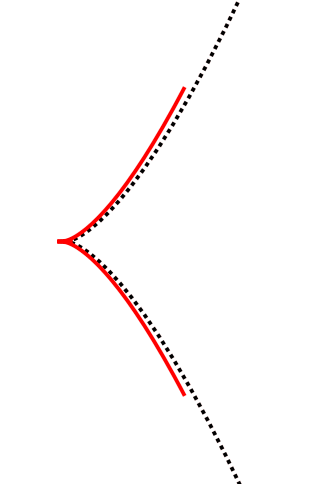}
    \caption{A curve $t\mapsto(t^2,t^3)$ (dashed line) and its singular evolutoid set parameterized by \eqref{eq:SEScuspEq}}
    \label{fig:sesCusp}
\end{figure}
\item It is enough to check that what the equations 
$$\det\left(\gamma_{\SES(\C)}'(s),\gamma_{\SES(\C)}''(s)\right)=0\text{ and }\det\left(\gamma_{\SES(\C)}'(s),\gamma_{\SES(\C)}'''(s)\right)=0$$ are equivalent to the stated conditions.
\item Let $f_{\C}:S^1\to\mathbb{R}^2$ be smooth with at most cusp singularities. The property
\begin{align}
\label{eq:SEScusps}    \rho_{\C}''(s)=0\quad\Rightarrow\quad\rho_{\C}'''(s)\neq 0
\end{align}
is equivalent to $\SES(\C)$ having at most cusp singularities.
Property \eqref{eq:SEScusps} can be viewed as the transversality of the map $j^3f:S^1\to J^3(S^1,\mathbb{R}^2)$ to the following submanifold of $J^3(S^1,\mathbb{R}^2)$:
\begin{align*}
    \big\{j^3g(s)\in J^3(S^1,\mathbb{R}^2)\,\big|\, \rho_g''(s)=0\}.
\end{align*}
By the Thom Transversality Theorem (e.g. see Theorem 4.9 in \cite{GGBook}) this property is generic. In a similar way one can show that having no undulation points and at most non-degenerate inflexion points by $\SES(\C)$ is also a generic property.
\end{enumerate}
\end{proof}

From now on, whenever we mention the word generic, we will mean the set described in Theorem \ref{ThmSESProperties}(v). Moreover, if we mention the word generic in the context of a fixed parameter $\alpha$, we will additionally mean the set which is the intersection of the sets described in Theorems \ref{ThmSESProperties}(v) and \ref{PropProperties}(v). Note that in a Baire space, which is the space of smooth maps between two smooth manifolds embedded with the Whitney $C^{\infty}$ topology, the intersection of two generic sets is still a generic set.

Let $\C$ be a smooth curve with at most cusp singularities having at least one inflexion point or undulation point. Note that each $\alpha$-evolutoid of $\C$ (for $\alpha\in(0,\pi)$) is unbounded since it has an asymptote which is the appropriate line through the inflexion point of $\C$. On the other hand, by Theorem \ref{ThmSESProperties}(i) we get that their singularities cannot appear arbitrarily far from $\C$ because the singular evolutoids set is a smooth curve. Figure \ref{fig:SESBean} shows, among others, a curve $\C$ with four inflexion points parameterized by 
\begin{align}\label{eq:BeanParam}
S^1\equiv[0,2\pi)\ni\varphi\mapsto (3-\cos 2\varphi)\cdot(\cos\varphi,\sin \varphi)\in\mathbb{R}^2
\end{align}
and its singular evolutoids sets. Note that $\SES(\C)$ has no singular points.

\begin{figure}[h]
    \centering
    \begin{subfigure}[h]{0.24\textwidth}
        \centering
        \includegraphics[width=\textwidth]{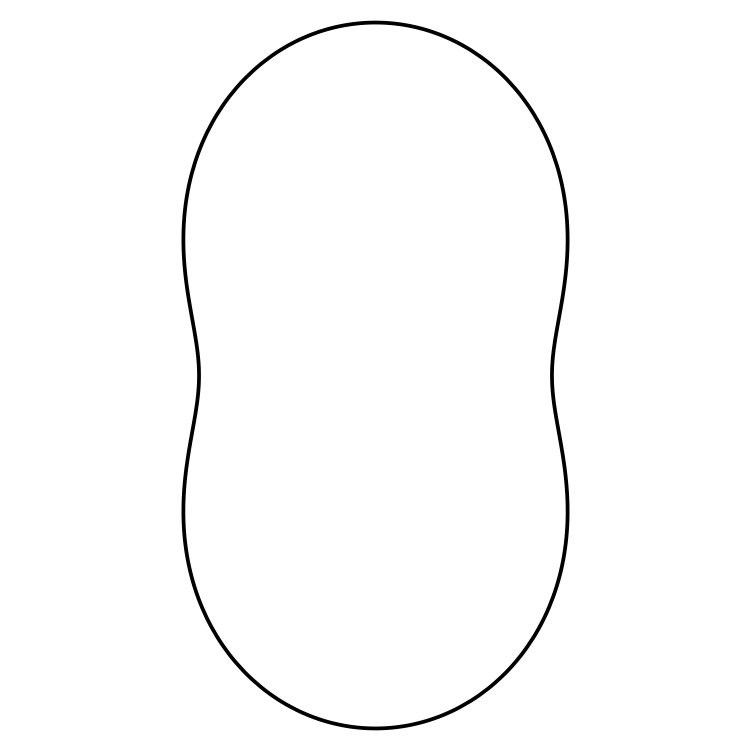}
        \caption{$\C$}
        \label{fig:SESBean}
    \end{subfigure}
    \hfill
    \begin{subfigure}[h]{0.24\textwidth}
        \centering
        \includegraphics[width=\textwidth]{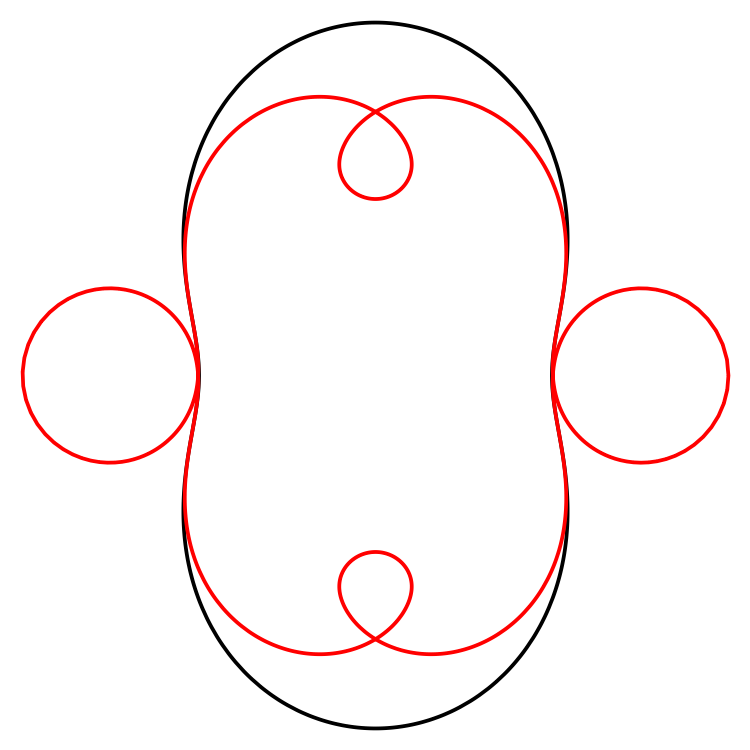}
        \caption{$\C$ and $\SES(\C)$}
        \label{fig:SESBean}
    \end{subfigure}
    \hfill
    \begin{subfigure}[h]{0.48\textwidth}
        \centering
        \includegraphics[width=\textwidth]{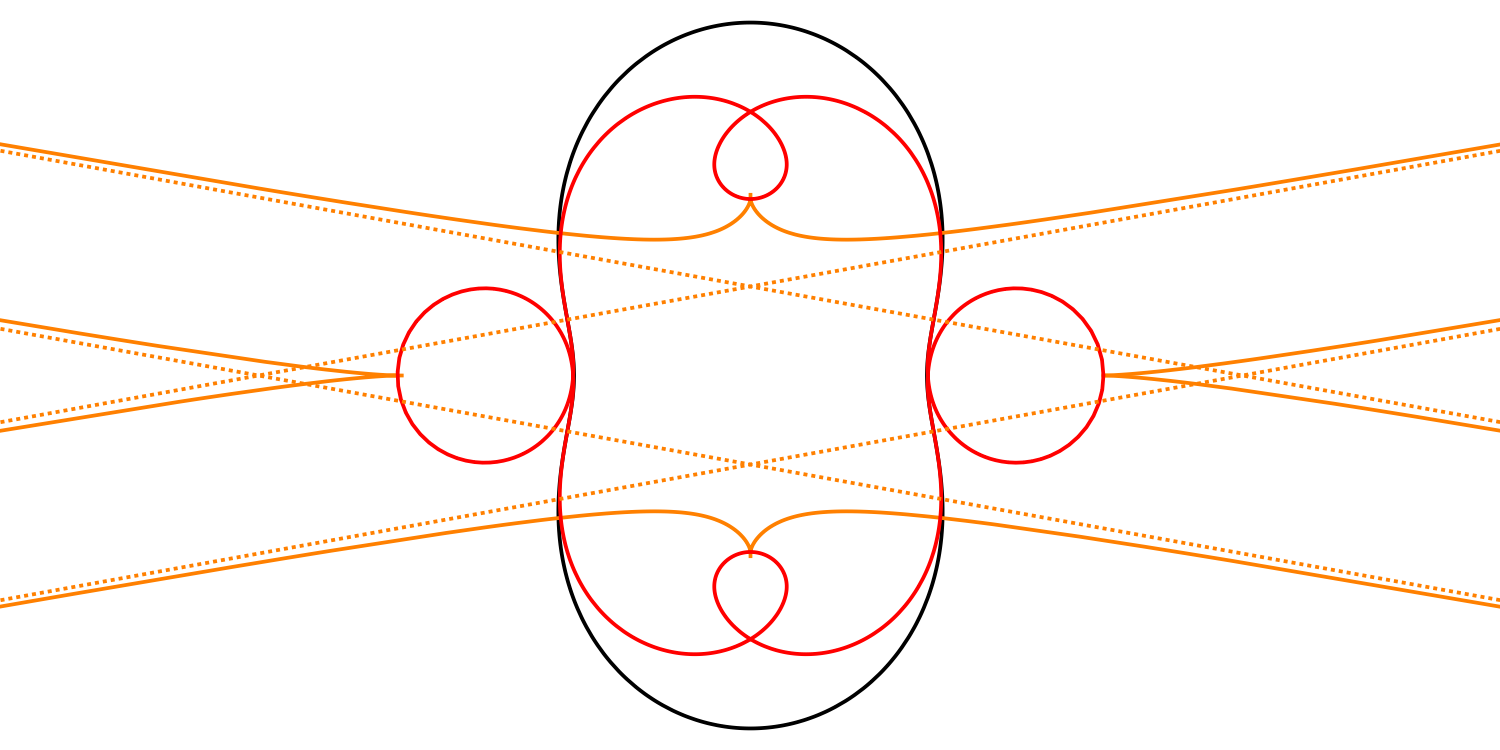}
        \caption{$\C$, $\SES(\C)$, and $\Ev_{\pi/2}(\C)$}
        \label{fig:SESBean}
    \end{subfigure}
    \caption{Sets for the curve $\C$ parameterized by \eqref{eq:BeanParam}}
    \label{fig:SESBean}
\end{figure}

%%%%%%%%%%%%%%%%%%%%%%%%%%%%%%%%%%%%%%%%%%%%%%%%%%%%%%%%%%%%%%
%%%%%%%%%%%%%%%%%%%%%%%%%%%%%%%%%%%%%%%%%%%%%%%%%%%%%%%%%%%%%%
%%%%%%%%%%%%%%%%%%%%%%%%%%%%%%%%%%%%%%%%%%%%%%%%%%%%%%%%%%%%%%

\section{Extended Evolutoids Front}

\noindent The \textit{extended euclidean space} is the space $\mathbb{R}^3_e=\mathbb{R}\times\mathbb{R}^2$ with coordinate $\alpha\in\mathbb{R}$ (called the \textit{time}) on the first factor and a projection $\pi:\mathbb{R}^3_e\to\mathbb{R}^2$ on the second factor denoted by $\pi(\alpha,\mathbbm{x})=\mathbbm{x}$.

\begin{defn}
Let $\C$ be a smoothly parameterized curve with at most cusp singularities. Then the \textit{extended evolutoids front} of $\C$ is the following set:
\begin{align}
    \EEF(\C):=\bigcup_{\alpha\in [0,\pi]}\,\{\alpha\}\times\Ev_{\alpha}(\C)\subset\mathbb{R}^3_e.
\end{align}
In other words, $\EEF(\C)$ is the union of all $\alpha$-evolutoids for $\alpha\in[0,\pi]$, each embedded into its own slice of the extended space.
\end{defn}

When $\C$ is a circle on the plane, then $\EEF(\C)$ is a curved double ‘cone', which is a smooth manifold with the non-singular projection $\pi$ everywhere, except its singular point, which projects to the center of the circle (see Figure \ref{fig:EEFCircle}).

\begin{figure}[h]
    \centering
    \begin{subfigure}[h]{0.30\textwidth}
        \centering
        \includegraphics[width=\textwidth]{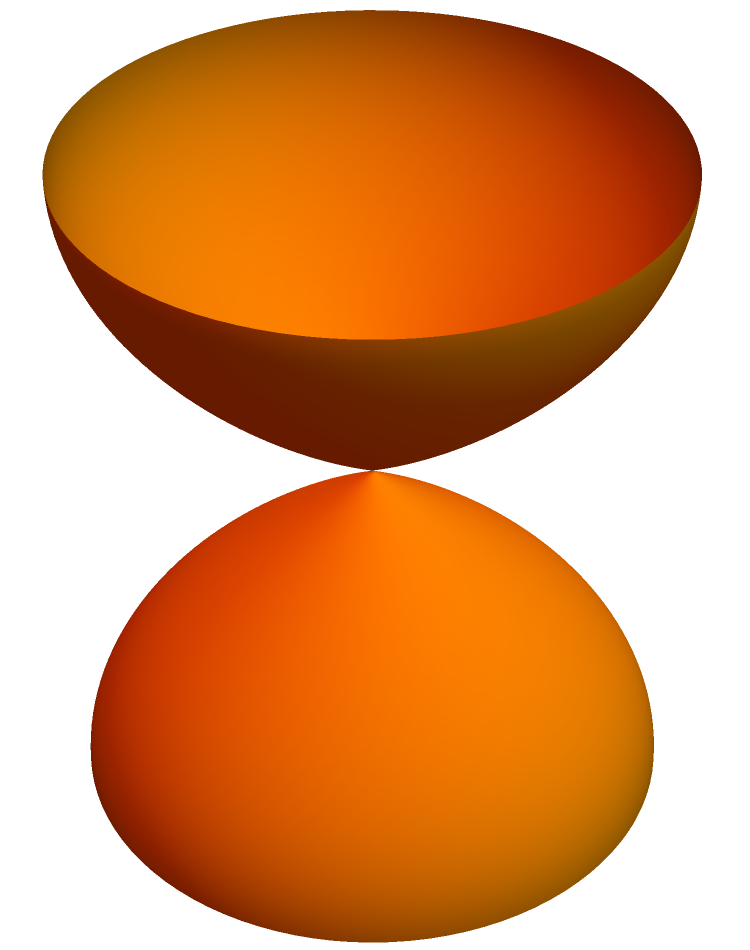}
        \caption{$o=100\%$}
        \label{fig:EEFCircle}
    \end{subfigure}
    \begin{subfigure}[h]{0.30\textwidth}
        \centering
        \includegraphics[width=\textwidth]{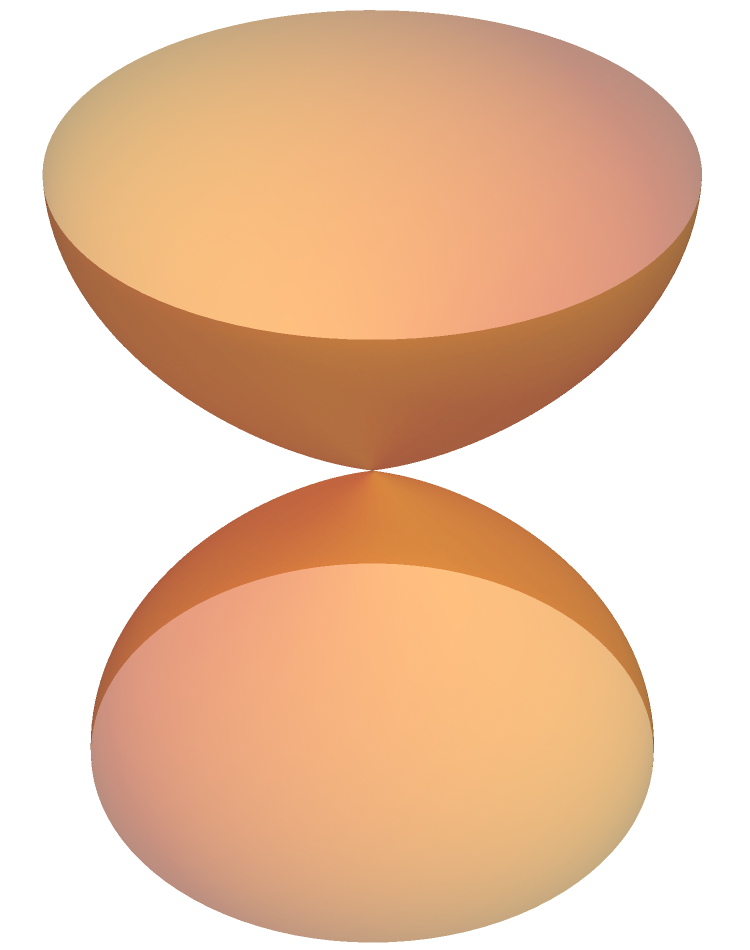}
        \caption{$o=50\%$}
        \label{fig:EEFCircle}
    \end{subfigure}
    \caption{The extended evolutoids front of a circle with different opacities $o$}
    \label{fig:EEFCircle}
\end{figure}

Let $\Sigma$ be a set of singular points of $\EEF$. By direct calculations we get the following proposition. 

\begin{prop}\label{PropProjSES}
Let $\C$ be a smooth curve with at most cusp singularities. Then
\begin{align*}
    \pi\big(\Sigma\left(\EEF(\C)\right)\big)=\SES(\C).
\end{align*}
\end{prop}

See Figures \ref{fig:Fig_eefoval1} and \ref{fig:Fig_eefhedghs} for examples of extended evolutoids fronts.

\begin{figure}[h]
    \centering
    \begin{subfigure}[h]{0.24\textwidth}
        \centering
        \includegraphics[width=\textwidth]{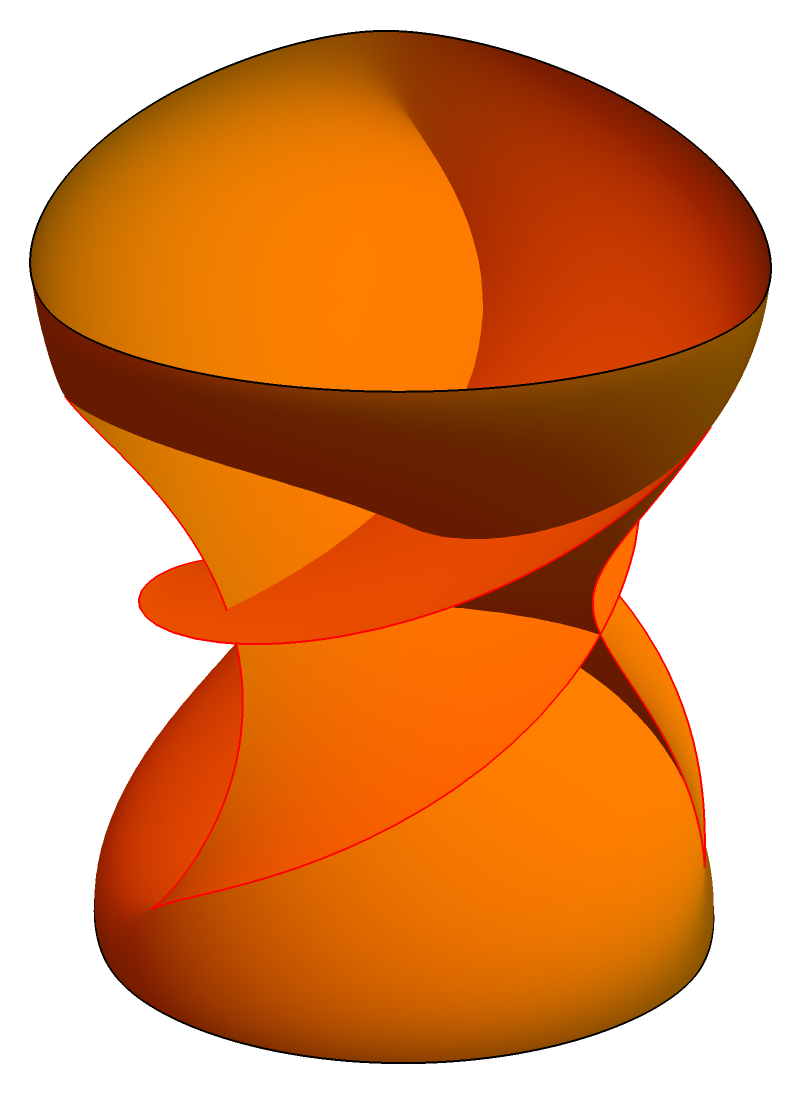}
        \caption{$o=100\%$}
        \label{fig:Fig_eefoval1}
    \end{subfigure}
    \hfill
    \begin{subfigure}[h]{0.24\textwidth}
        \centering
        \includegraphics[width=\textwidth]{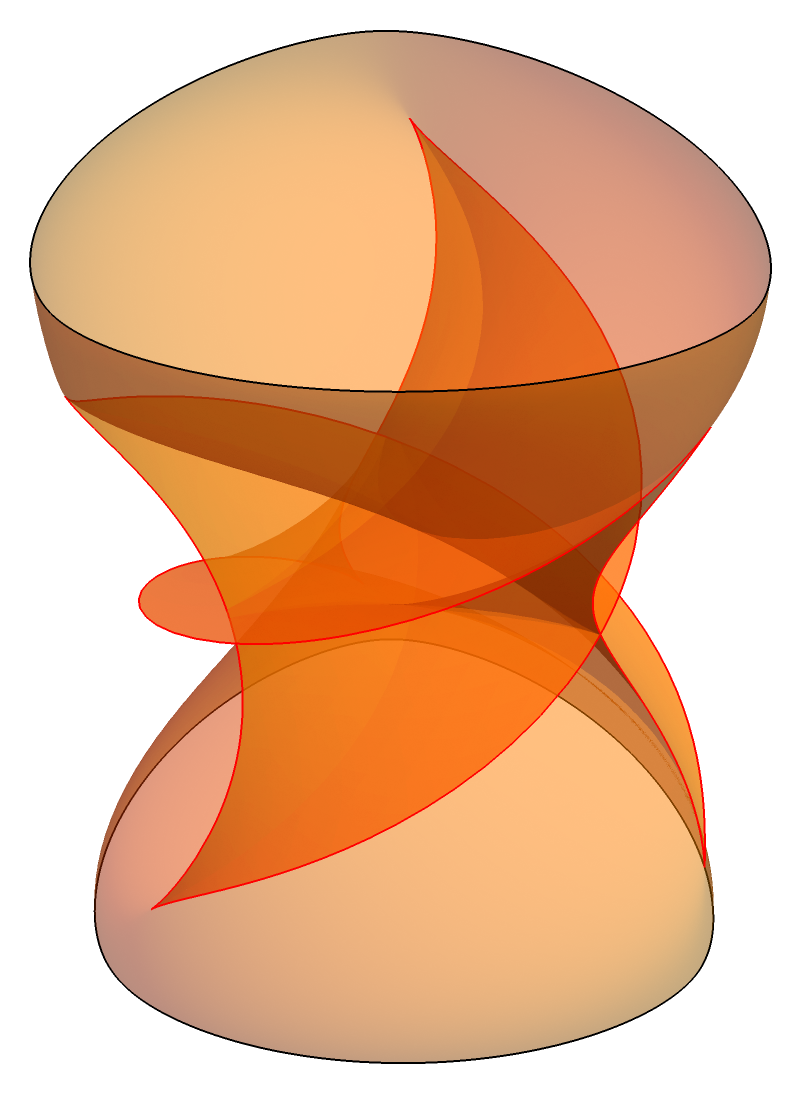}
        \caption{$o=50\%$}
        \label{fig:Fig_eefoval1}
    \end{subfigure}
    \hfill
    \begin{subfigure}[h]{0.24\textwidth}
        \centering
        \includegraphics[width=\textwidth]{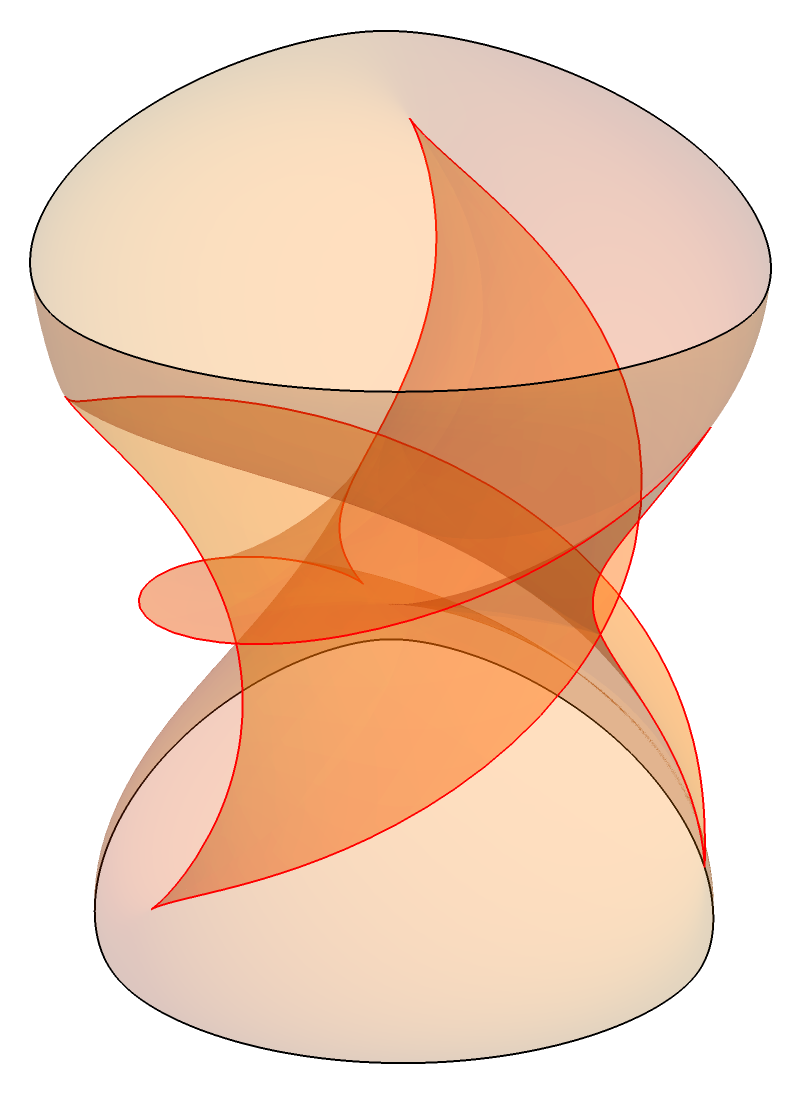}
        \caption{$o=25\%$}
        \label{fig:Fig_eefoval1}
    \end{subfigure}
    \hfill
        \begin{subfigure}[h]{0.24\textwidth}
        \centering
        \includegraphics[width=\textwidth]{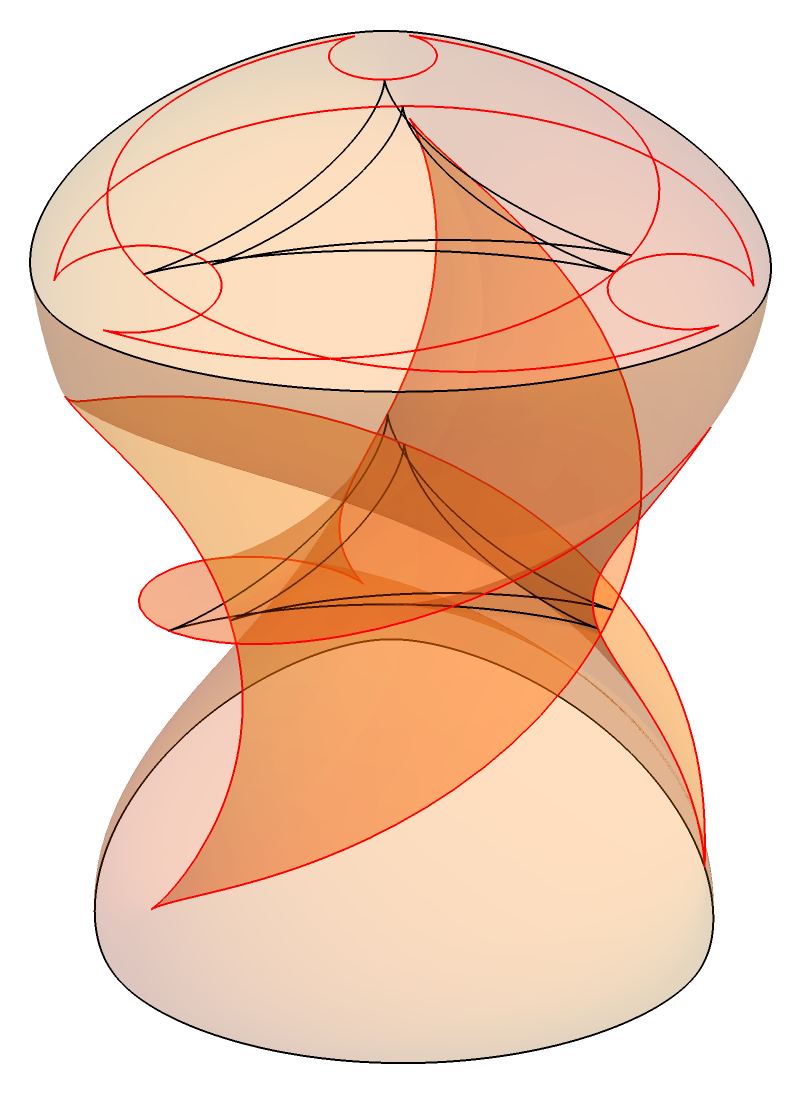}
        \caption{$o=25\%$}
        \label{fig:Fig_eefoval1}
    \end{subfigure}
    \caption{The extended evolutoids front of an oval $\C$ with support function $p(\theta)=40+3\cos 3t-\sin 2t$ with different opacities $o$. (D) also shown $\{\pi\}\times\SES(\C)$ and $\{\pi/2,\pi\}\times\Ev_{\pi/2}(\C)$}
    \label{fig:Fig_eefoval1}
\end{figure}

\begin{figure}[h]
    \centering
    \begin{subfigure}[h]{0.24\textwidth}
        \centering
        \includegraphics[width=\textwidth]{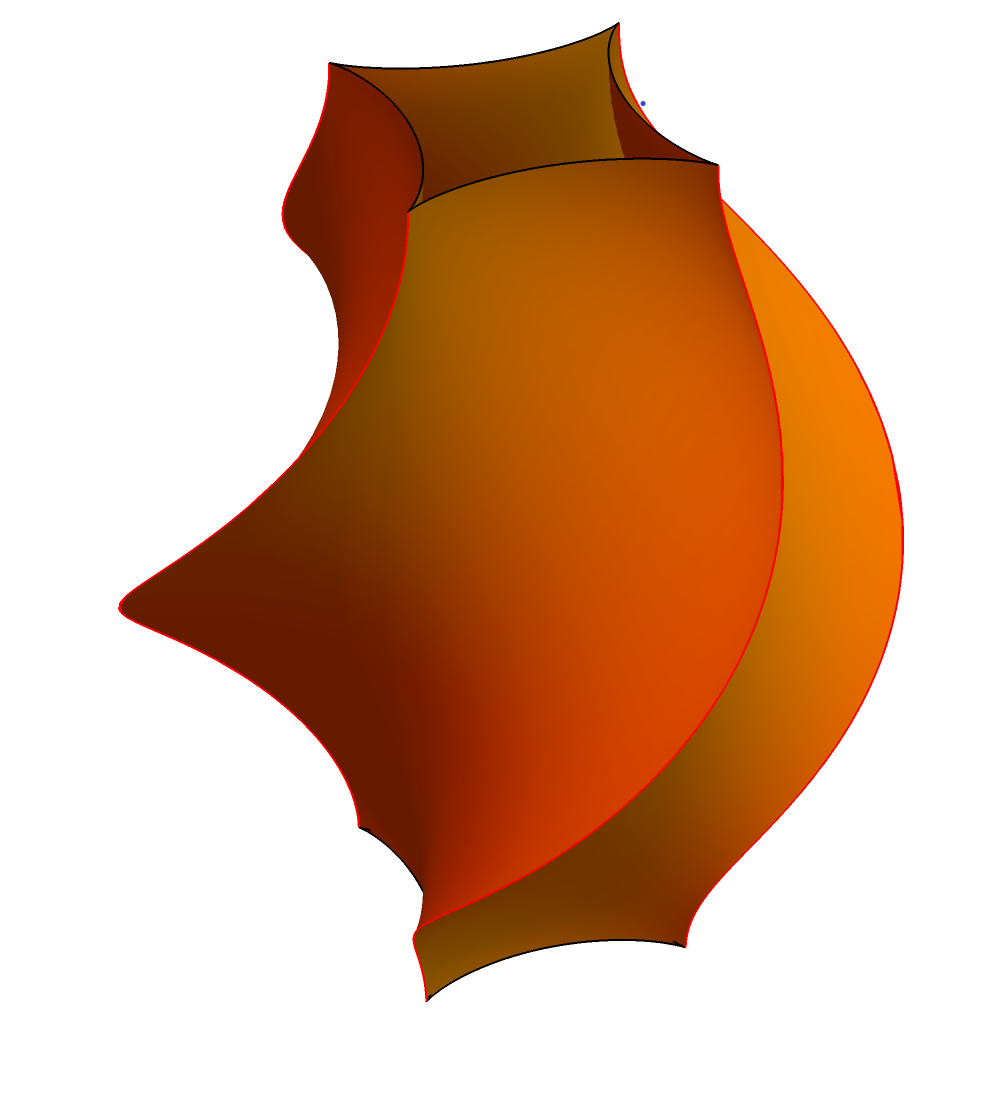}
        \caption{$k=2$}
        \label{fig:Fig_eefhedghs}
    \end{subfigure}
    \hfill
    \begin{subfigure}[h]{0.24\textwidth}
        \centering
        \includegraphics[width=\textwidth]{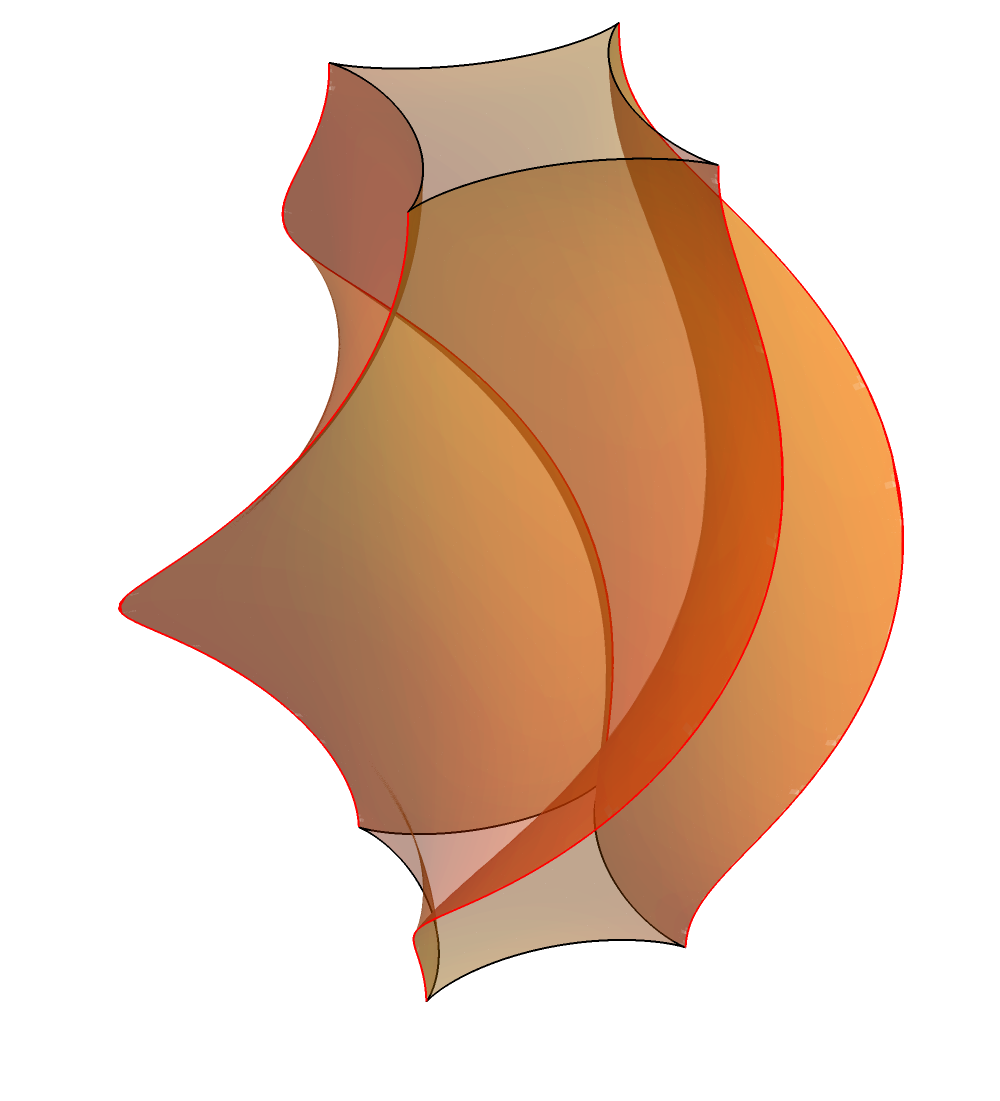}
        \caption{$k=2$}
        \label{fig:Fig_eefhedghs}
    \end{subfigure}
    \hfill
        \begin{subfigure}[h]{0.24\textwidth}
        \centering
        \includegraphics[width=\textwidth]{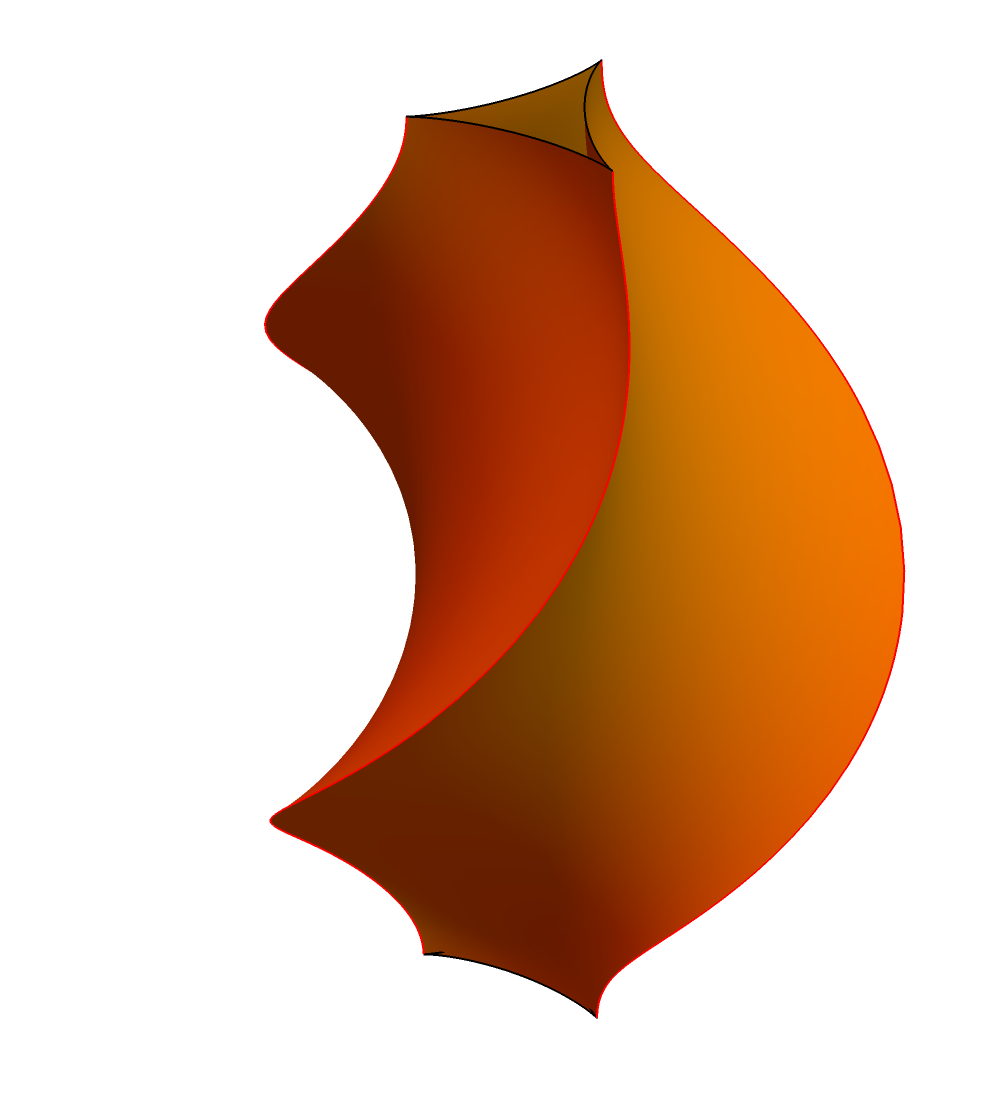}
        \caption{$k=3$}
        \label{fig:Fig_eefhedghs}
    \end{subfigure}
    \hfill
    \begin{subfigure}[h]{0.24\textwidth}
        \centering
        \includegraphics[width=\textwidth]{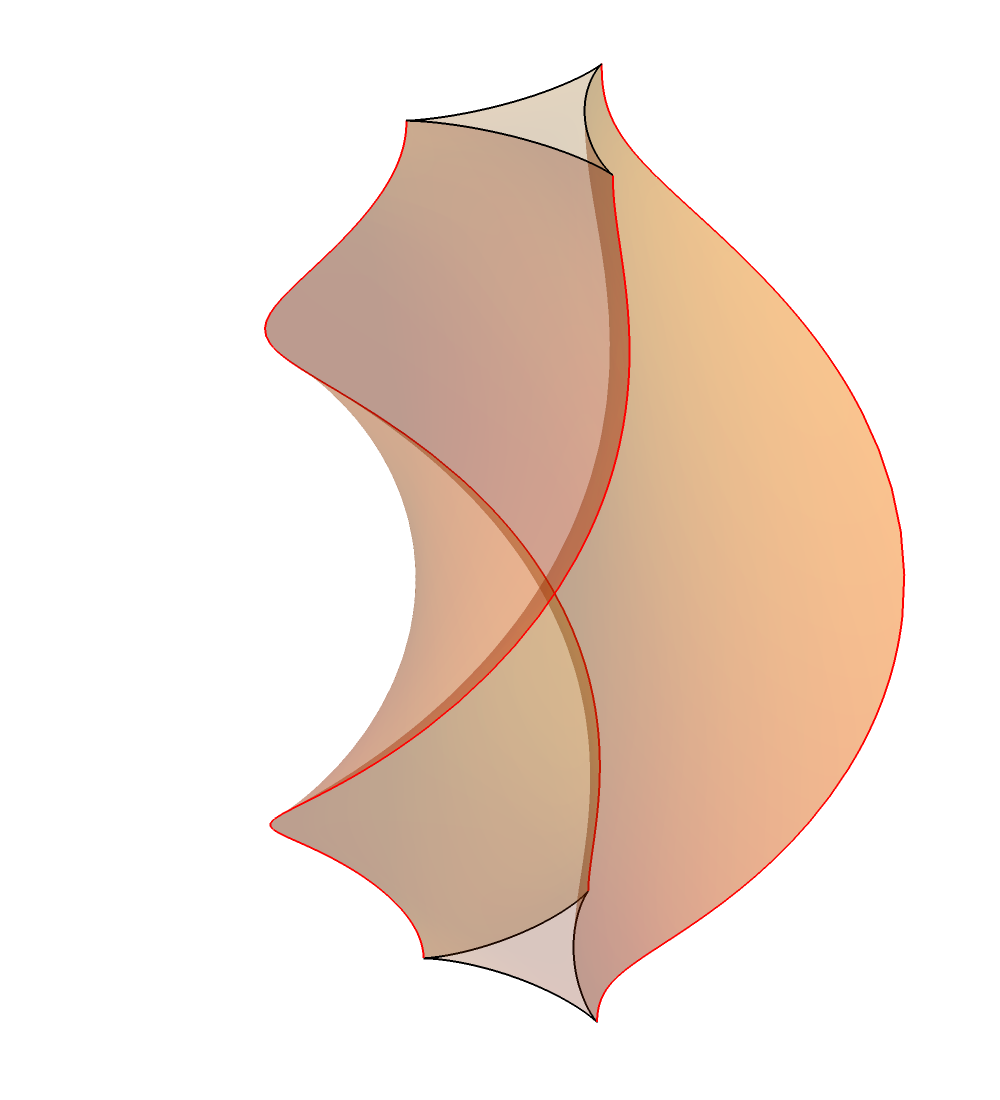}
        \caption{$k=3$}
        \label{fig:Fig_eefhedghs}
    \end{subfigure}
    
        \begin{subfigure}[h]{0.24\textwidth}
        \centering
        \includegraphics[width=\textwidth]{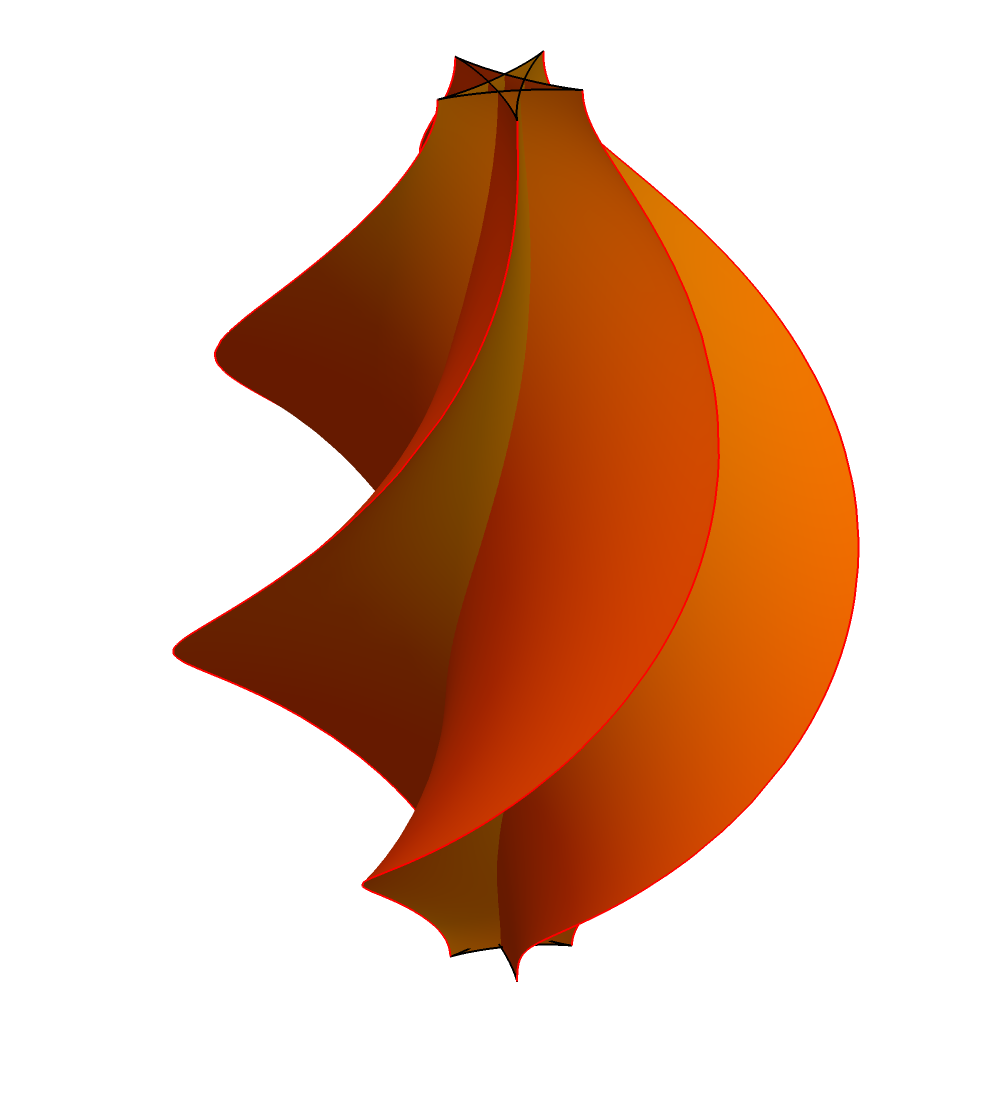}
        \caption{$k=5$}
        \label{fig:Fig_eefhedghs}
    \end{subfigure}
    \hfill
    \begin{subfigure}[h]{0.24\textwidth}
        \centering
        \includegraphics[width=\textwidth]{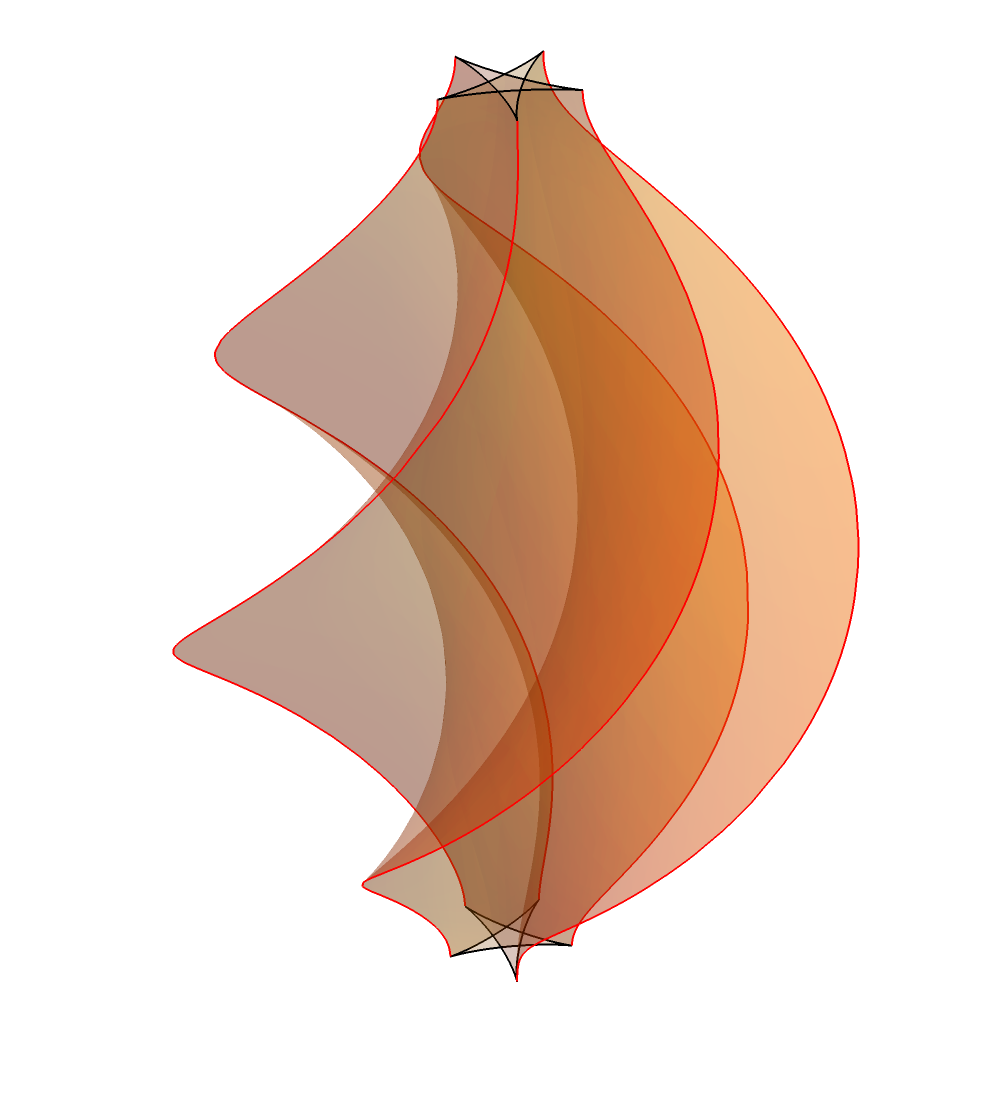}
        \caption{$k=5$}
        \label{fig:Fig_eefhedghs}
    \end{subfigure}
    \hfill
        \begin{subfigure}[h]{0.24\textwidth}
        \centering
        \includegraphics[width=\textwidth]{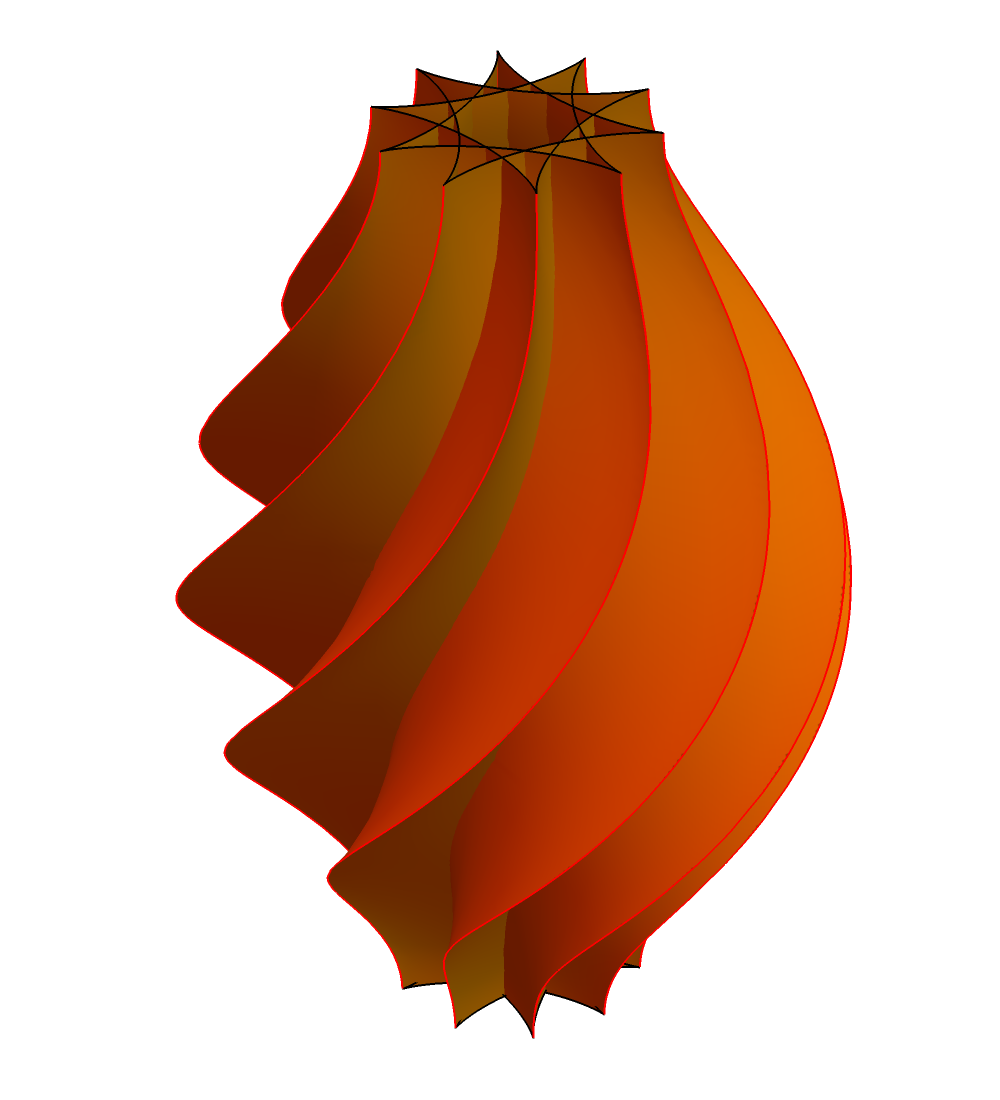}
        \caption{$k=2.5$}
        \label{fig:Fig_eefhedghs}
    \end{subfigure}
    \hfill
    \begin{subfigure}[h]{0.24\textwidth}
        \centering
        \includegraphics[width=\textwidth]{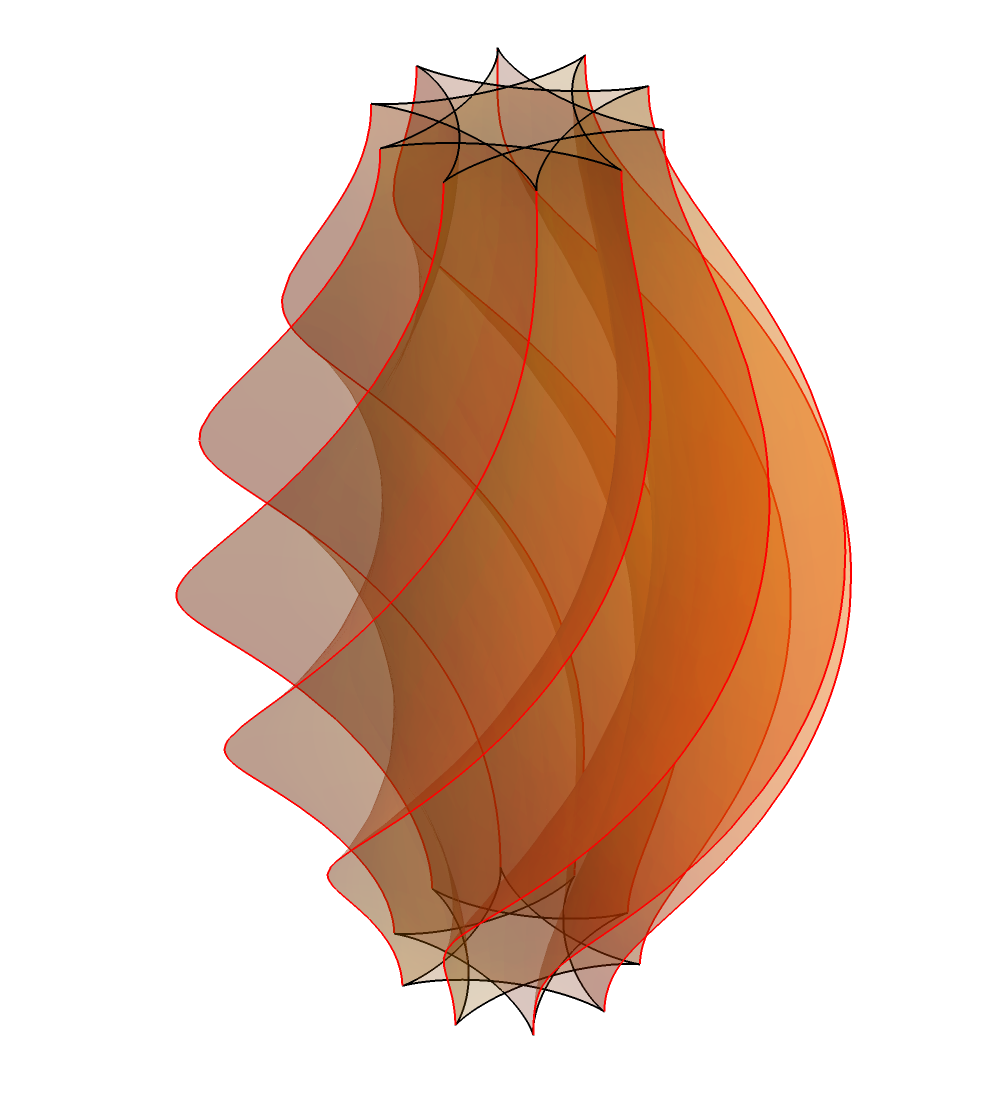}
        \caption{$k=2.5$}
        \label{fig:Fig_eefhedghs}
    \end{subfigure}

    \caption{The extended evolutoids fronts for hedgehogs with support functions $p_k(\theta)=\sin k\theta$ with opacities $100\%$ and $25\%$ for $k=2, 3, 5, 2.5$, respectively}
    \label{fig:Fig_eefhedghs}
\end{figure}

In \cite{GiblinWarder} the extended evolutoids front is studied by the name \textit{discriminant set} (locally and for $\alpha\in\left[0,\frac{\pi}{2}\right]$). It was shown (see Proposition 5.2) in \cite{GiblinWarder} that $\EEF(\C)$ is a cuspidal edge at $(\alpha,s)$ if and only if
\begin{align*}
    \rho_{\C}(s)=-\cot\alpha,\quad\rho_{\C}'(s)\neq 0,
\end{align*}
and it is a swallowtail if and only if 
\begin{align*}
    \rho_{\C}(s)=-\cot\alpha,\quad\rho_{\C}'(s)=0,\quad\rho_{\C}''(s)\neq 0,
\end{align*}
where $'$ denotes the derivative with respect to the arc length parameter.

By Theorem \ref{ThmSESProperties} we can say that generically $\EEF$ has only cuspidal edges and swallowtails as singularities.

Now we recall some definitions related to coherent tangent bundles (for details see \cite{DZ-GaussBonnet}). Let $M$ be a $2$-dimensional compact and oriented surface. A $5$-tuple $(M,\mathcal{E},\left<\cdot,\cdot\right>,D,\psi)$ is called a \textit{coherent tangent bundle} over $M$, where $\mathcal{E}$ is an orientable vector bundle over $M$ of rank $2$, $\left<\cdot,\cdot\right>$ is a metric, $D$ is a metric connection on $(\mathcal{E},\left<\cdot,\cdot\right>)$ and $\psi$ is a bundle homomorphism $\psi:TM\to\mathcal{E}$ such that 
$$D_X\psi(Y)-D_Y\psi(X)=\psi([X,Y]),$$
where $X,Y$ are any smooth vector fields on $M$, and $[X,Y]$ is the Lie bracket of $X$ and $Y$. Let $\mathcal{E}_p$ denote the fiber of $\mathcal{E}$ at $p$. The point $p$ is singular if $\psi|_{T_pM}:T_pM\to\mathcal{E}_p$ is not a bijection. Let $\Sigma$ denote the set of singular points on $M$. Let $(U;u,v)$ be a positively oriented local coordinate system on $M$. The \textit{signed area form} (respectively \textit{unsigned area form} is $\d\hat{A}=\lambda_\psi\d u\wedge\d v$ (respectively $\d A=|\lambda_\psi|\d u\wedge \d v$), where
$$\lambda_\psi:=\mu\left(\dfrac{\partial}{\partial u},\dfrac{\partial}{\partial v}\right),$$
where $\mu\in\mathrm{Sec}(\mathcal{E}^*\wedge\mathcal{E}^*)$ is a smooth non-vanishing skew-symmetric bilinear section such that for any orthonormal frame $e$ on $\mathcal{E}$ we get $\mu(e)=\pm 1$. The function $\lambda_\psi$ is called the \textit{signed area density function} on $U$. The set of singular points on $U$ is $\Sigma\cap U=\{p\in U\colon \lambda_\psi(p)=0$. Furthermore, let us define 
$$M^+:=\big\{p\in M\setminus\Sigma\colon \lambda_\psi(p)>0\big\},
\quad 
M^-:=\big\{p\in M\setminus\Sigma\colon \lambda_\psi(p)<0\big\},$$
i.e. non-singular point $p$ is in $M^\pm$ if and only if $\d\hat{A}_p=\pm\d A_p$.
We say that a singular point $p\in\Sigma$ is \textit{non-degenerate} if $\d\lambda_\psi$ does not vanish at $p$. If $p$ is a non-degenerate singular point, then there exists a neighborhood $U$ of $p$ such that $\Sigma\cap U$ is a regular curve, which is called the \text{singular curve}. Since $p$ is non-degenerate, the rank of $\psi_p$ is $1$. The null direction is the direction of the kernel of $\psi_p$. Let $\eta$ be the smooth (non-vanishing) vector field along the singular curve which gives the null direction. Let $\sigma(t)$ be a singular curve such that $\sigma(0)=p$ is a non-degenerate singular point. The point $p$ is called an $A_2$-\textit{point} (or an \textit{intrinsic cuspidal edge}) if the null direction at $p$ is transversal to the singular direction at $p$. The point $p$ is called an $A_3$-\textit{point} (or an \textit{intrinsic swallowtail}) if the point $p$ is not an $A_2$-point and 
$$\dfrac{\d}{\d t}\left(\sigma'(t)\wedge\eta(t)\right)\big|_{t=0}\neq 0.$$
Let $p\in\Sigma$ which is not an $A_2$-point. Then $p$ is called a \textit{peak} if there exists a coordinate neighborhood $(U;u,v)$ of $p$ such that
\begin{itemize}
    \item if $q\in(\Sigma\cap U)\setminus\{p\}$, then $q$ is an $A_2$-point;
    \item the rank of the map $\psi_p:T_pM\to\mathcal{E}_p$ at $p$ is equal to $1$,
    \item the set $\Sigma\cap U$ consists of finitely many $C^1$-regular curves emanating from $p$.
\end{itemize}

Let $[0,1)\ni t\mapsto\gamma(t)\in M$ be a $C^1$ regular curve on $M$ such that $\gamma(0)=p$. Then, the $\mathcal{E}$-initial vector of $\gamma$ at $p$ is the following limit
$$\Psi_\gamma:=\lim_{t\to 0^+}\dfrac{\psi(\dot{\gamma}(t))}{|\psi(\dot{\gamma}(t))|}\in\mathcal{E}_p$$
if it exists. Let $\gamma_1$, $\gamma_2$ be two $C^1$-regular curves emanating from $p$ such that the $\mathcal{E}$-initial vectors of $\gamma_1$ and $\gamma_2$ exist. Then the angle $\arccos\left(\left<\Psi_{\gamma_1},\Psi_{\gamma_2}\right>\right)$ is called the \textit{angle between the initial vectors} of $\gamma_1$ and $\gamma_2$ at $p$. A domain $\Omega\subset U\subset M^+$ is a \textit{positive singular sector} at $p\in U$ if $\Omega\cap\Sigma=\varnothing$ and the boundary of $\Omega\cap U$ consists of $\sigma_1,\sigma_2$ and the boundary of $U$, where $\sigma_1,\sigma_2$ are curves in $U$ starting at $p$ so that both are singular curves or one of them is a singular curve and the other one is in $\partial M$.
If $p$ is a singular point, then $\alpha_+(p)$ is the sum of all interior angles of positive singular sectors at $p$. 

Let $\sigma(t)$ be a $C^2$-regular curve on $M$. We assume that if $\sigma(t)\in\Sigma$, then $\dot{\sigma}$ is transversal to the null direction at $\sigma(t)$. Then, the image of $\psi(\dot{\sigma}(t))$ does not vanish and then we take a parameter $\tau$ such that
$$\left\|\psi\left(\dfrac{\d}{\d\tau}\sigma(\tau)\right)\right\|\equiv 1.$$
Let $n(\tau)$ be a section of $\mathcal{E}$ along $\sigma(\tau)$ such that $\left\{\psi\left(\frac{\d}{\d\tau}\sigma(\tau)\right),n(\tau)\right\}$ is a positive orthonormal frame. Then
$$\kappa_g(\tau):=\left<D_{\frac{\d}{\d\tau}}\psi\left(\frac{\d}{\d\tau}\sigma(\tau)\right),n(\tau)\right>=\nu\left(\psi\left(\frac{\d}{\d\tau}\sigma(\tau)\right), D_{\frac{\d}{\d\tau}}\psi\left(\frac{\d}{\d\tau}\sigma(t)\right)\right)$$
is called the $\mathcal{E}$-geodesic curvature of $\sigma$ which gives the geodesic curvature of $\sigma$ with respect to the orientation of $\mathcal{E}$. Now let's assume that $\sigma$ is a singular curve consisting of $A_2$-points. Take a null vector field $\eta(\tau)$ along $\sigma(\tau)$ such that $\left\{\frac{\d}{\d\tau}\sigma(\tau),\eta(\tau)\right\}$ is a positively oriented field along $\sigma$ for each $\tau$. Then, the \textit{singular curvature function} is defined by
$$\kappa_s(\tau):=\mathrm{sgn}\left(\d\lambda_\psi(\eta(\tau))\right)\cdot\kappa_g(\tau).$$
For the properties of singular curvature see especially \cite{SUY2,SUY3}. Now let $U\subset M$ be a domain and let $\{e_1,e_2\}$ denote a positive orthonormal frame field on $U$. There exists a unique $1$-form $\omega$ on $U$ such that
$$D_Xe_1=-\omega(X)e_2,\quad D_Xe_2=\omega(X)e_1,$$
where $X$ is a smooth vector field on $U$ ($D$ is a metric connection). Furthermore, note that
$$\d\omega=K\d\hat{A}=\left\{\begin{array}{ll} K\d A & \text{on }M^+, \\ -K\d A &\text{on }M^-,\end{array}\right.$$
where $K$ is the Gaussian curvature of the first fundamental form $\d s^2$ (for details see \cite{SUY,SUY2}).

From now on, we will assume that $\C$ is a generic $k$-hedgehog, where $k$ is some half-integer. Notice that if $k$ is not an integer, then $M$ is orientable on the double covering of $\C$.  Let $p$ be the support function of $\C$ and let $f_{\C}$ be the parameterization of $\C$ in terms of polar-tangential coordinates (see \eqref{eqSuppParamter}). Furthermore, we set:
\begin{align}
    M&:=[0,\pi]\times S^1,\\
    \label{eqEEFParameter}
    M&\ni (\alpha,\theta)\mapsto\ff(\alpha,\theta):=\big(\alpha,f_{\alpha}(\theta)\big)\in\EEF(\C)\subset\mathbb{R}^3_e,\\ 
    M&\ni (\alpha,\theta)\mapsto\nu(\alpha,\theta):=\dfrac{\big(-\rho_{\C}(\theta)\sin\alpha, \mathbbm{n}_{\C}(\theta+\alpha)\big)}{\sqrt{1+\rho^2_{\C}\sin^2\alpha}}\in S^2,
\end{align}
where $f_{\alpha}(\theta):=f_{\C}(\theta)+\rho_{\C}(\theta)\sin\alpha\cdot\tt_{\C}(\theta+\alpha)$ is a parameterization of $\Ev_{\alpha}(C)$. Note that in polar tangential coordinates $\tt_{\C}(\theta)=(-\sin\theta,\cos\theta)$ and $\nn_{\C}(\theta)=(-\cos\theta,-\sin\theta)$.

The map $\ff$ is a front since $(\ff,\nu)$ is a Legendrian immersion. Hence, the fiber at $p\in M$ of the coherent tangent bundle $\mathcal{E}^{\ff}$ over $M$ is
\begin{align*}
    \mathcal{E}^{\ff}_{p}:=\left\{X\in T_{\ff(p)}\mathbb{R}^3\, \big|\, \left<X,\nu(p)\right>=0\right\}.
\end{align*}
Since 
\begin{align*}\lambda(\alpha,\theta)=\big(\rho_{\C}(\theta)\cos\alpha+\rho'_{\C}(\theta)\sin\alpha\big)\sqrt{1+\rho_{\C}^2(\theta)\sin^2\alpha},
\end{align*}
the set of singular points $\Sigma\subset M$ is given by the equation $\rho_{\C}(\theta)\cos\alpha+\rho_{\C}'(\theta)\sin\alpha=0$. Observe that 
\begin{align*}
    M^-&:=\left\{(\alpha,\theta)\in M\, \big|\,\lambda(\alpha,\theta)<0\right\}=\left\{(\alpha,\theta)\in M\, \big|\,\rho_{\C}(\theta)\cos\alpha+\rho_{\C}'(\theta)\sin\alpha<0\right\}, \\ 
    M^+&:=\left\{(\alpha,\theta)\in M\, \big|\,\lambda(\alpha,\theta)>0\right\}=\left\{(\alpha,\theta)\in M\, \big|\,\rho_{\C}(\theta)\cos\alpha+\rho_{\C}'(\theta)\sin\alpha>0\right\}.
\end{align*}
Note that if $\C$ is a rosette, then the inequality $\rho_{\C}(\theta)\cos\alpha+\rho_{\C}'(\theta)\sin\alpha<0$ is equivalent to $\alpha<\alpha(\theta):=\textrm{arccot}\left(-\rho'_{\C}(\theta)/\rho_{\C}(\theta)\right)$. Furthermore, if the function $\alpha(\theta)$ has a local maximum (respectively a local minimum), then the point $(\alpha(\theta),\theta)$ is a negative peak (respectively a positive peak) -- see Figure \ref{fig:domains}(A). Note that in the case of singular hedgehogs $\C$, a local maximum doesn't have to be a negative peak -- see Figure \ref{fig:domains}(B). It can be verified that if $\C$ is a generic hedgehog, then the front $\ff$ admits at most peak singularities, $\Sigma$ is transversal to $\partial M$, and any point $p\in\Sigma\cap\partial M$ is a null singular point since the null direction is tangent to the boundary of $\M$ at $p$ (for details see \cite{DZ-GaussBonnet}).

\begin{figure}[h]
    \centering
    \begin{subfigure}[h]{0.48\textwidth}
        \centering
        \includegraphics[width=\textwidth]{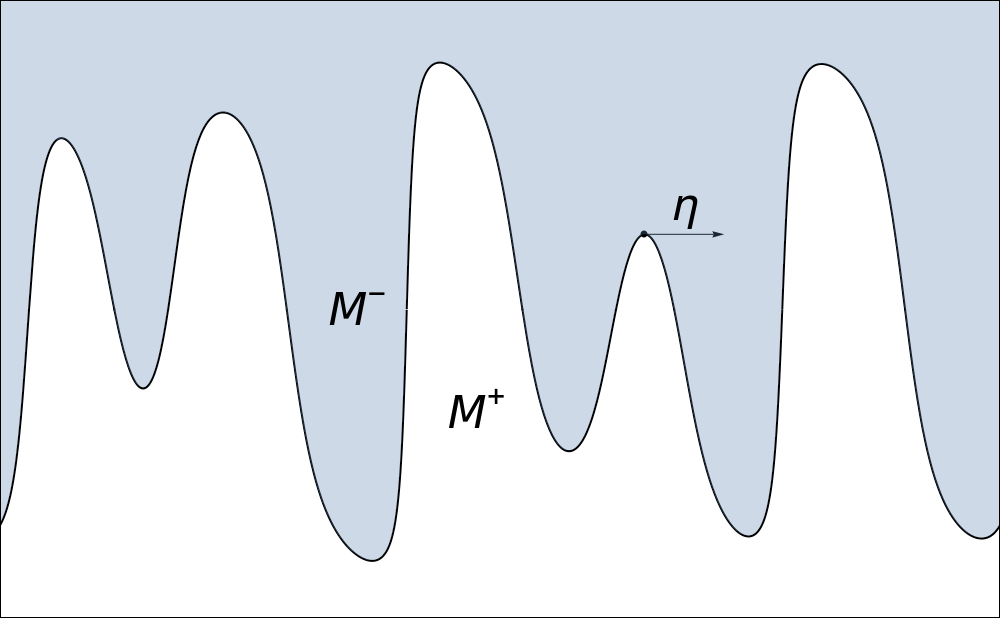}
        \caption{$p(\theta)=20+\sin 3\theta+0.5\sin 2\theta+0.25\cos 5\theta$}
        \label{fig:domains}
    \end{subfigure}
    \hfill
    \begin{subfigure}[h]{0.48\textwidth}
        \centering
        \includegraphics[width=\textwidth]{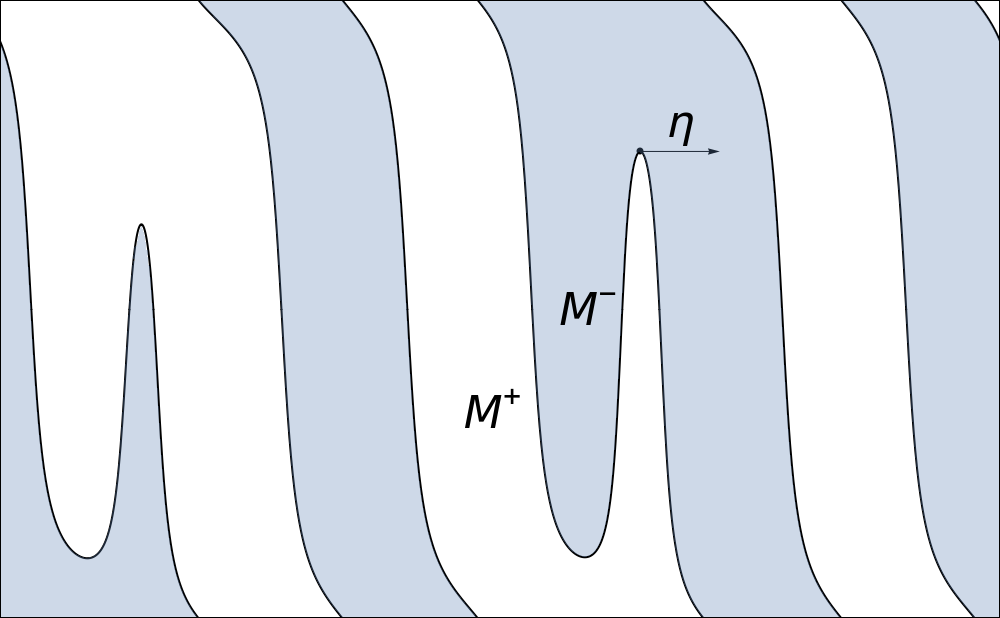}
        \caption{$p(\theta)=\sin 3\theta+0.2\cos 5\theta+0.125\cos 2\theta$}
        \label{fig:domains}
    \end{subfigure}

    \caption{Domains of extended evolutoids fronts of hedgehogs of support functions $p(\theta)$, singular curves and null vectors}
    \label{fig:domains}
\end{figure}

\begin{lem}\label{LemmaGeodCurv}
Let $\C$ be a generic hedgehog and let $\alpha\in[0,\pi]$. The $\mathcal{E}^{\ff}$-geodesic curvature of a curve $\{\alpha\}\times S^1$ in $M$ at a non-singular point is equal to
\begin{align}
    \label{eqGeodesicCurv}\hat{\kappa}_{g,\alpha}(\theta)=-\dfrac{\rho_{\C}(\theta)\sin\alpha}{\sqrt{1+\rho_{\C}^2(\theta)\sin^2\alpha}\cdot\left|\rho_{\C}(\theta)\cos\alpha+\rho'_{\C}(\theta)\sin\alpha\right|}.
\end{align}
Furthermore,
\begin{align}
    \label{eqGeodesicCurvDTau} \hat{\kappa}_{g,\alpha}\,\d\tau = -\dfrac{\rho_{\C}(\theta)\sin\alpha}{\sqrt{1+\rho_{\C}^2(\theta)\sin^2\alpha}}\,\d\theta,
\end{align}
where $\d\tau$ is the arc length measure of the image of $\{\alpha\}\times S^1$ by $\ff$.
\end{lem}
\begin{proof}
Let $\gamma_{\alpha}(\theta):=\ff(\alpha,\theta)$. Then \eqref{eqGeodesicCurv} and \eqref{eqGeodesicCurvDTau} follow from the formula
\begin{align*}
    \hat{\kappa}_{g,\alpha}(\theta)=\dfrac{\det\big(\gamma'_{\alpha}(\theta),\gamma''_{\alpha}(\theta),\nu(\alpha,\theta)\big)}{\left|\gamma'_{\alpha}(\theta)\right|^3}.
\end{align*}
\end{proof}

\begin{lem}\label{LemSingCurvCalc}
Let $\C$ be a generic hedgehog. Then the singular curvature of a singular curve at $(\mathrm{arccot}\left(-\rho_{\C}'(\theta)/\rho_{\C}(\theta), \theta\right)$ is
\begin{align}
    \kappa_{s}(\theta)=\tfrac{\mathrm{sgn}\left(\rho_{\C}'^2(\theta)-\rho_{\C}(\theta)\rho_{\C}''(\theta)\right)\cdot\Big(\rho_{\C}^6(\theta)+\rho_{\C}'^4(\theta)-\rho_{\C}^4(\theta)\left(\rho_{\C}'^2(\theta)-1\right)+2\rho_{\C}^3(\theta)\rho_{\C}''(\theta)+2\rho_{\C}^5(\theta)\rho_{\C}''(\theta)\Big)}{\left(\rho_{\C}'^2(\theta)-\rho_{\C}(\theta)\rho_{\C}''(\theta)\right)\left(1+\rho_{\C}^2(\theta)\right)^{3/2}\left(\rho_{\C}^2(\theta)+\rho_{\C}'^2(\theta)\right)\sqrt{\rho_{\C}^2(\theta)+\rho_{\C}^4(\theta)+\rho_{\C}'^2(\theta)}}.
\end{align}
Furthermore,
\begin{align}
    \kappa_s\,\d\tau=\frac{\rho_{\C}^6(\theta)+\rho_{\C}'^4(\theta)-\rho_{\C}^4(\theta)\left(\rho_{\C}'^2(\theta)-1\right)+2\rho_{\C}^3(\theta)\rho_{\C}''(\theta)+2\rho_{\C}^5(\theta)\rho_{\C}''(\theta)}{\left(1+\rho_{\C}^2(\theta)\right)\left(\rho_{\C}^2(\theta)+\rho_{\C}'^2(\theta)\right)\sqrt{\rho_{\C}^2(\theta)+\rho_{\C}^4(\theta)+\rho_{\C}'^2(\theta)}}\,\d\theta,
\end{align}
where $\d\tau$ is the arc length measure of the image of the singular curve by $\ff$.
\end{lem}
\begin{proof}
Since $\eta$ is spanned by $\frac{\d}{\d\alpha}$ and $(\gamma',\eta)$ forms a positive oriented frame, where $\gamma$ is a singular curve, we obtain that \begin{align*}
    \eta=\overline{\mathrm{sgn}}\left(-\rho'^2_{\C}(\theta)+\rho_{\C}(\theta)\rho''_{\C}(\theta)\right)\dfrac{\d}{\d\alpha},
\end{align*}
where $\overline{\mathrm{sgn}}(t)=\mathrm{sgn}(t)$ for $t\neq 0$ and $\overline{\mathrm{sgn}}(0)=1$. Therefore, $\mathrm{sgn}\left(\d\lambda(\eta)\right)=1$. Then we calculate $\kappa_s(\theta)$ and $\kappa_s\,\d\tau$ using the following formula
\begin{align*}
    \kappa_s(\theta)=\mathrm{sgn}\left(\d\lambda(\eta)\right)\dfrac{\det\big(\hat\gamma'(\theta),\hat\gamma''(\theta),\nu\big)}{\left|\hat\gamma'(\theta)\right|^3},
\end{align*}
where $\hat\gamma=\ff\circ\gamma$ and $\gamma\subset M$ is the singular curve.
\end{proof}

\begin{lem}\label{LemGaussCurvCalc}
The Gaussian curvature of $\ff$ at a non-singular point is $(\alpha,\theta)$
\begin{align}
    K(\alpha,\theta)=\dfrac{\rho_{\C}(\theta)\cos\alpha-\rho_{\C}'(\theta)\sin\alpha}{\big(1+\rho_{\C}^2(\theta)\sin^2\alpha\big)^2\big(\rho_{\C}(\theta)\cos\alpha+\rho_{\C}'(\theta)\sin\alpha\big)}.
\end{align}
Furthermore,
\begin{align*}
    K\,\d A=\dfrac{\big(\rho_{\C}(\theta)\cos\alpha-\rho_{\C}'(\theta)\sin\alpha\big)\mathrm{sgn}\big(\rho_{\C}(\theta)\cos\alpha+\rho_{\C}'(\theta)\sin\alpha\big)}{\big(1+\rho_{\C}^2(\theta)\sin^2\alpha\big)^{3/2}}\,\d\alpha\wedge\d\theta.
\end{align*}
\end{lem}
\begin{proof}
It is enough to use the following formula for the absolute area form of $\ff$
\begin{align*}
    \d A=|\lambda|\,\d\alpha\wedge\d\theta
\end{align*}
and the Gaussian curvature $K$.
\end{proof}

\begin{thm}\label{ThmGBforFront}
Let $\C$ be a generic hedgehog and $\EEF$ its extended evolutoids front. Then 
\begin{align}
    \label{eq:GBEvolutoids}\int_{M}K\,\d A=-2\int_{\Sigma}\kappa_s\,\d\tau.
\end{align}
\end{thm}
\begin{proof}
Since $M=[0,1]\times S^1$, we get that $\chi(M)=0$. By Lemma \ref{LemmaGeodCurv}, the total geodesic curvatures of $\{0\}\times S^1$ and  $\{\pi\}\times S^1\subset M$, namely:
\begin{align*}
    \int_{\partial M\cap M^+}\hat\kappa_g\,\d\tau, 
    \int_{\partial M\cap M^-}\hat\kappa_g\,\d\tau,
\end{align*}
are both zero. 

Let $p\in\mathrm{null}(\Sigma\cap\partial M)$. Then if $p=(0,\theta_0)$ or $p=(\pi,\theta_0)$ we obtain that $\rho(\theta_0)=0$ and $\rho'(\theta_0)\neq 0$ (by the genericity of $\C$). By direct calculations, we find that
\begin{align*}
    \alpha_+(p)=\arccos\left(\pm\dfrac{\rho_{\C}(\theta_0)\big(\rho_{\C}^2(\theta_0)-\rho_{\C}'^2(\theta_0)\big)}{\sqrt{1+\rho_{\C}^2(\theta_0)}\big(\rho_{\C}^2(\theta_0)+\rho_{\C}'^2(\theta_0)\big)}\right)=\dfrac{\pi}{2}.
\end{align*}
Therefore, by the Gauss Bonnet Formula \eqref{GBplusformula} we obtain \eqref{eq:GBEvolutoids}, which ends the proof.
\end{proof}

The second Gauss Bonnet Formula \eqref{GBminusformula} is trivially satisfied (for details see \cite{DZ-GaussBonnet}).

By Lemma \ref{LemSingCurvCalc} and Lemma \ref{LemGaussCurvCalc} we obtain the following corollary of Theorem \ref{ThmGBforFront}.

\begin{cor}\label{CorTheEquation}
    Let $\C$ be a generic $k$-hedgehog. Then
    \begin{align*}
        &\int_{[0,\pi]\times[0,2\tilde{k}\pi]}\dfrac{\big(\rho_{\C}(\theta)\cos\alpha-\rho_{\C}'(\theta)\sin\alpha\big)\mathrm{sgn}\big(\rho_{\C}(\theta)\cos\alpha+\rho_{\C}'(\theta)\sin\alpha\big)}{\big(1+\rho_{\C}^2(\theta)\sin^2\alpha\big)^{3/2}}\,\d\alpha\wedge\d\theta
        =\\ 
        =-2&\int_0^{2\tilde{k}\pi}\frac{\rho_{\C}^6(\theta)+\rho_{\C}'^4(\theta)-\rho_{\C}^4(\theta)\left(\rho_{\C}'^2(\theta)-1\right)+2\rho_{\C}^3(\theta)\rho_{\C}''(\theta)+2\rho_{\C}^5(\theta)\rho_{\C}''(\theta)}{\left(1+\rho_{\C}^2(\theta)\right)\left(\rho_{\C}^2(\theta)+\rho_{\C}'^2(\theta)\right)\sqrt{\rho_{\C}^2(\theta)+\rho_{\C}^4(\theta)+\rho_{\C}'^2(\theta)}}\,\d\theta,
    \end{align*}
    where $\tilde{k}=k$ if $k$ is an integer, otherwise: $\tilde{k}=2k$.
\end{cor}

\bibliographystyle{amsalpha}

\end{document}